\documentclass{amsart}
\usepackage{graphicx}
\usepackage{graphicx,esint}
\usepackage{amssymb,amscd,amsthm,amsxtra}
\usepackage{latexsym}
\usepackage{epsfig}

\newtheorem{thm}{Theorem}[section]
\newtheorem{cor}[thm]{Corollary}
\newtheorem{lem}[thm]{Lemma}
\newtheorem{prop}[thm]{Proposition}
\theoremstyle{definition}
\newtheorem{defn}[thm]{Definition}
\theoremstyle{remark}
\newtheorem{rem}[thm]{Remark}
\numberwithin{equation}{section}

\newcommand{\R}{\mathbb R}
\newcommand{\be}{\begin{equation}}
\newcommand{\ee}{\end{equation}}

\newcommand{\eps}{\varepsilon}
\newcommand{\ph}{\varphi}
\newcommand{\ov}{\overline}
\newcommand{\Om}{\Omega}
\newcommand{\p}{\partial}

\newcommand{\comment}[1]{}

\begin{document}

\title[A localization theorem]{A localization theorem and boundary regularity for a class of degenerate Monge Ampere equations}

\author{Ovidiu Savin}

\begin{abstract}
We consider degenerate Monge-Ampere equations of the type
$$\det D^2 u= f \quad \mbox{in $\Om$}, \quad \quad f \sim \, d_{\p \Om}^\alpha \quad \mbox{near $\p \Om$,}$$
where $d_{\p \Om}$ represents the distance to the boundary of the domain $\Om$ and $\alpha>0$ is a positive power.
We obtain $C^2$ estimates at the boundary under natural conditions on the boundary data and the right hand side.
Similar estimates in two dimensions were obtained by J.X. Hong, G. Huang and W. Wang in \cite{HHW}.
\end{abstract}

\maketitle

\section{Introduction}

In this paper we discuss boundary regularity for solutions to degenerate Monge-Ampere equations of the type
$$\det D^2 u= f \quad \mbox{in $\Om$}, \quad \quad f \sim \, \, d_{\p \Om}^\alpha \quad \mbox{near $\p \Om$,}$$
where $d_{\p \Om}$ represents the distance to the boundary of a convex domain $\Om$ and $\alpha>0$ is a positive power.

Boundary estimates for the Monge-Ampere equation in the nondegenerate case $f\in C(\ov \Om)$, $f>0$, were obtained starting with the works of Ivockina \cite{I}, Krylov \cite{K}, Caffarelli-Nirenberg-Spruck \cite{CNS} (see also \cite{C,TW,W}). The general strategy for the $C^2$ estimates in the nondegenerate case is to obtain first a bound by above for the second derivatives on $\p \Om$, and then to use the equation and bound all the pure second derivatives by below. When $f=0$ on $\p \Omega$ this bound cannot hold since some second derivative becomes $0$. In this paper we show that, under general conditions on the data, in a neighborhood of $\p \Om$ only one second derivative tends to 0 and all tangential pure second derivatives are continuous and bounded by below away from 0. The difficulty in proving this result lies in the fact that the tangential pure second derivatives are only subsolutions for the linearized operator, and therefore it is not clear whether or not such a lower bound is satisfied.
In the case of two dimensions J.X. Hong, G. Huang and W. Wang in \cite{HHW} used that the tangential second derivative is in fact a solution to an elliptic equation and showed that $u\in C^2$ up to the boundary.

In this paper we study the geometry of boundary sections in the degenerate case when $f$ behaves in a neighborhood of $\p \Om$ as a positive power of the distance to $\p \Om$. We use the compactness methods developed in \cite{S1} where a localization theorem for boundary sections of solutions to the Monge-Ampere equation was obtained. In Theorem \ref{T1} we show that a localization theorem holds also in the degenerate case, and it states that boundary sections have the shape of half-ellipsoids. We achieve this by reducing the problem to the study of tangent cones for solutions to degenerate Monge-Ampere equations that have a singularity on $\p \Omega$. Then we use the ideas from \cite{S2} where the regularity of such tangent cones was investigated for the classical Monge-Ampere equation.

Before we state our main results we recall the notion for a function to be $C^2$ at a point. We say that $u$ is $C^2$ at $x_0$ if there exists a quadratic polynomial $Q_{x_0}$ such that, in the domain of definition of $u$, $$u(x)=Q_{x_0}(x) + o(|x-x_0|^2).$$
 Throughout this paper we refer to a linear map $A$ of the form
 $$Ax=x+\tau x_n, \quad \quad \mbox{with} \quad \tau \cdot e_n=0,$$
as a {\it sliding along} $x_n=0$. Notice that the map $A$ is the identity map when is restricted to $x_n=0$ and it becomes a translation of vector $s \tau$ when is restricted to $x_n=s$ .

Let $\Omega$ be a bounded convex domain such that $\p \Om$ is $C^{1,1}$ at the origin, that is $0 \in \p \Om$ and
\begin{equation}\label{01}
\Om \subset \{x_n>0\},\quad \mbox{ and $\Om$ has an interior tangent ball at the origin.}
\end{equation}
We are interested in the behavior near the origin of a convex solution $u\in C(\ov \Om)$ to the equation
\begin{equation}\label{02}
\det D^2 u=g(x) \, d_{\p \Om}^\alpha, \quad \quad \quad \alpha>0,
\end{equation}
where $g$ is a nonnegative function that is continuous at the origin, $g(0)>0$.

Our main theorem is the following pointwise $C^2$ estimate at the boundary (see also Theorem \ref{T2} for a more precise quantitative version).

\begin{thm}\label{T01}
Let $\Om$, $u$ satisfy \eqref{01}, \eqref{02} above. Assume that
$$ u(0)=0, \quad \nabla u(0)=0, \quad u=\ph \quad \mbox{on $\p \Om$},$$
and the boundary data $\ph$ is $C^2$ at $0$, and it separates quadratically away from $0$.

Then $u$ is $C^2$ at $0$. Precisely, there exists a sliding $A$ along $x_n=0$ and a constant $a>0$ such that
$$u(Ax)=Q_0(x')+ a x_n^{2+\alpha} + o(|x'|^2+x_n^{2+\alpha}),$$
where $Q_0$ represents the quadratic part of the boundary data $\ph$ at the origin.
\end{thm}

If the hypotheses above hold and $\p \Om \in C^2$, $\ph \in C^2$, $g \in C^\beta$ in a neighborhood of $0$, then $u \in C^2(\ov \Om \cap B_\delta)$, for some small $\delta>0$ (see Theorem \ref{T2.1}). Here we require $g \in C^\beta$ only to guarantee the $C^2$ regularity at interior points close to $\p \Om$.

It is worth remarking that the $C^2$ estimate of Theorem \ref{T01} does not hold for harmonic functions or solutions to the classical Monge-Ampere equation. In these cases we need stronger assumptions on $\p \Om$ and $\ph$, i.e. to be $C^{2,Dini}$ at the origin.

In a subsequent work we intend to use Theorem \ref{T01} and perturbations arguments to obtain $C^{2,\beta}$ and higher order estimates when the data $\p \Om$, $\ph$, $g$ is more regular.

Our second result which is closely related to Theorem \ref{T01} is a Liouville theorem for degenerate solutions to Monge-Ampere equations defined in half-space.

\begin{thm}\label{T02}
Assume $u \in C(\ov{\R^n_+})$ satisfies
\begin{equation}\label{03}
\det D^2 u=x_n^\alpha, \quad \quad u(x',0)=\frac 12 |x'|^2.
\end{equation}
If there exists $\eps>0$ small such that $u =O( |x|^{3+\alpha-\eps})$ as $|x| \to \infty$, then
$$u(Ax)=bx_n+\frac 12 |x'|^2 + \frac{x_n^{2+\alpha}}{(1+\alpha)(2+\alpha)},  $$
for some sliding $A$ along $x_n=0$, and some constant $b$.
\end{thm}

We remark that Theorem \ref{T02} holds also for $\alpha=0$. The theorem states that solutions to \eqref{03} that grow at a power less than $|x|^{3+\alpha}$ at $\infty$ are unique modulo additions of $c \, x_n$ and domain deformations given by slidings along $x_n=0$. Clearly, both transformations leave \eqref{03} invariant.
 The growth condition at infinity is necessary since $$\frac{x_1^2}{2(1+x_n)} + \frac 12 (x_2^2+...+x_{n-1}^2)+ \frac{x_n^{2+\alpha}}{(1+\alpha)(2+\alpha)} + \frac{x_n^{3+\alpha}}{(2+\alpha)(3+\alpha)} $$
satisfies also \eqref{03}.

 In the two dimensional case Theorem \ref{T02} follows easily after performing a partial Legendre transform in the $x_1$ direction. Then the problem reduces to the classification of solutions to a linear equation defined in half-space. However this approach does not seem to work in higher dimensions.

Theorem \ref{T01} applies when the right hand side $f$, which may depend also on $u$ and $\nabla u$, is expected to behave as a power of the distance to $\p \Om$. For example we obtain $C^2$ estimates up to the boundary for solutions to the eigenvalue problem for Monge-Ampere equation which was first investigated by Lions in \cite{L}.

\begin{thm}\label{T03}
Assume $\p \Om \in C^2$ is uniformly convex and $u \in C(\ov \Om)$ satisfies  $$(\det D^2 u)^\frac 1 n= \lambda |u| \quad \quad \mbox{in $\Om$},  \quad \quad u=0 \quad \mbox{on $\p \Om$.}$$
 Then $u \in C^2(\ov \Om)$.
\end{thm}

In two dimensions Theorem \ref{T03} was obtained in \cite{HHW}.

The paper is organized as follows. In Section \ref{s2} we introduce some notation and state our main results, the localization Theorem \ref{T1} and the quantitative $C^2$ estimate Theorem \ref{T2}. Most of the paper is devoted to the proof of the localization Theorem \ref{T1}. In Section \ref{s3} we deal with some general properties of boundary sections. In Section \ref{s4} we use compactness and reduce Theorem \ref{T1} to Theorem \ref{T3} which deals with estimates of boundary sections for a class of solutions with discontinuities on $\p \Om$. In Section \ref{s5} we obtain two Pogorelov type estimates for solutions to certain Monge-Ampere equations. We use these estimates in Section \ref{s6} where we complete the proof of Theorem \ref{T3}. In Section \ref{s7} we prove a Liouville theorem from which Theorem \ref{T2} follows. Finally is Section \ref{s8} we prove Theorems \ref{T02} and \ref{T03}.

\section{Statement of main results}\label{s2}

We introduce some notation. We denote points in $\R^n$ as
$$x=(x_1,...,x_n)=(x',x_n), \quad \quad \quad x' \in \R^{n-1}.$$
We denote by $B_r(x)$ the ball of radius $r$ and center $x$, and by $B_r'(x')$ the ball in $\R^{n-1}$ of radius $r$ and center $x'$.

Given a convex function $u$ defined on a convex set $\ov \Om$, we denote by $S_h(x_0)$ the section centered at $x_0$ and height $h>0$,
$$S_h(x_0):=\{x \in \ov \Om | \quad  u(x) < u(x_0)+\nabla u(x_0) \cdot (x-x_0) +h \}.$$
We denote for simplicity $S_h=S_h(0)$, and sometimes when we specify the dependence on the function $u$ use the notation $S_h(u)=S_h$.

Throughout the paper we think of the constants $n$, $\alpha$ and $\mu$ as being fixed. We refer to all positive constants depending only $n$, $\alpha$ and $\mu$ as {\it universal constants} and we denote them by $c$, $C$, $c_i$, $C_i$. The dependence of various constants also on other parameters like $\rho$ and $\rho'$ will be denoted by $c(\rho, \rho')$.

Our assumptions are the following (we assume $\rho$, $\rho'$ are small positive constants).

First we assume $\Omega$ is $C^{1,1}$ at the origin, that is

\

H1) $\Om$ is an open convex set , $0 \in \p \Om$, $$\Om \subset \{ x_n > 0\} \cap B_{1/\rho},$$ and $\Om$ has an interior tangent ball of radius $\rho$ at the origin.

\

Let $x_{n+1}=0$ be the tangent plane for a continuous convex function $u:\ov \Om \to \R$ at the origin, that is

\

H2) $u \ge 0$, $u(0)=0$, $\nabla u(0)=0$ in the sense that $x_{n+1}=t x_n$ is not a supporting plane for the graph of $u$ at $0$ for any $t >0$.

\

We assume that $u$ separates on $\p \Omega$ quadratically away from its tangent plane in a neighborhood of $0$. Precisely

\

H3) For some $\eps_0 \in (0,\frac 14)$ we have $$(1-\eps_0)\ph(x') \le u(x) \le (1+\eps_0) \ph(x') \quad \quad \mbox{for all} \quad x \in \p \Omega \cap B_{\rho/2},$$ with $\ph(x')$ a function of $n-1$ variables satisfying $$\mu^{-1}I \ge D^2_{x'} \ph \ge \mu \, \, I,$$ and also at the points on $\p \Omega$ outside $B_{\rho/2}$ we assume
$$u(x) \ge \rho' \quad \mbox{on} \quad \p \Om \cap \{ x_n \le \rho\} \setminus B_{\rho/2}.$$

\

We assume that the Monge-Ampere measure of $u$ near $0$ behaves as $d_{\p \Om}^\alpha$ where $d_{\p \Om}(x)$ denotes the distance from $x$ to $\p \Om$ i.e.,

\

H4) $$(1-\eps_0) d_{\p \Om}^\alpha \le \det D^2 u \le (1+\eps_0)d_{\p \Om}^\alpha \quad \quad \mbox{in} \quad B_\rho \cap \Omega, $$
and $$\det D^2 u \le 1/ \rho' \quad \quad \mbox{in} \quad \{x_n < \rho\} \cap \Omega.$$

Our localization theorem states that if $u$ satisfies the hypotheses above then the sections $S_h$ of $u$ at the origin
are equivalent, up to a sliding along $x_n=0$, to the sections of the function $|x'|^2 + x_n^{2+\alpha}$.

\begin{thm} [Localization Theorem] \label{T1}
Assume H1, H2, H3, H4 are satisfied. If $\eps_0$ is sufficiently small, universal, then
$$k \, A \mathcal E_h \, \cap \ov \Om \, \subset S_h \, \subset k^{-1} \, A \mathcal E_h \, \cap \ov \Om  \quad \quad \mbox{for all $h < c(\rho, \rho')$,}$$
where $$ \mathcal E_h:=\{ |x'|^2+x_n^{2+\alpha} < h\},$$
and $A$ is a {\it sliding} along $x_n=0$ i.e.
$$Ax=x + \tau x_n, \quad \quad \tau=(\tau_1,\tau_2,..,\tau_{n-1},0), \quad |\tau|\le C(\rho,\rho').$$
The constant $k$ above is universal, that is depends only on $n$, $\alpha$ and $\mu$, and $c(\rho,\rho')$, $C(\rho,\rho')$ depend on the universal constants and $\rho$, $\rho'$.

\end{thm}

\begin{rem}\label{r0}
The conclusion can be stated as $$c(|x'|^2 + x_n^{2+\alpha}) \le u(Ax) \le C(|x'|^2 + x_n^{2+\alpha}),$$
in a neighborhood of the origin where $c$, $C$ are universal constants. Equivalently we can say that there exists a sliding $A$ such that $A^{-1} S_h$ is equivalent to an ellipsoid of axes parallel to the coordinate axes and of lengths $h^{1/2},h^{1/2}, \ldots, h^{1/2}, h^{1/(2+\alpha)}$.
\end{rem}

\begin{rem}
If $\det D^2 u=d_{\p \Om}^\alpha$ and $\p \Omega\in C^{1,1}$ in a neighborhood of $0$ then Theorem \ref{T1} provides bounds by above and below for the tangential (to $\p \Om$) second derivatives in a neighborhood of $0$. The conclusion of Theorem \ref{T1} can be viewed as a boundary $C^{1,1}$ estimate by below written in terms of the sections $S_h$ rather than using second derivatives.
\end{rem}

The localization theorem for the nondegenerate case $\alpha=0$ holds if $\det D^2 u$ is only bounded away from $0$ and $\infty$ (see \cite{S1}). When $\alpha>0$ the hypothesis that $g=d_{\p \Om}^{-\alpha} \det D^2 u$ has small oscillation is in fact optimal. It is possible to construct a counterexample for Theorem \ref{T1} in two dimensions if we allow $g$ to be only bounded. However in this case we obtain a pointwise $C^{1,\gamma}$ estimate (see Proposition \ref{p0}.)

Our second theorem provides a pointwise $C^2$ estimate for solutions $u$ as above in the case when the boundary data is $C^2$.

\begin{thm}\label{T2}
Assume $u$ satisfies the hypotheses of Theorem \ref{T1} with $$\ph(x')=\frac 12 |x'|^2.$$
For any $\eta>0$ there exists $\eps_0$ depending on $\eta$, $\alpha$ and $n$, and a sliding $A$ along $x_n=0$ such that
$$(1-\eta) A \, \, S_h(U_0) \subset S_h(u) \subset (1+\eta) A \, \, S_h(U_0) $$
for all $h < c(\eta,\rho,\rho')$ where $U_0$ is the particular solution
$$U_0(x):=\frac 12 |x'|^2 + \frac{x_n^{2+\alpha}}{(1+\alpha)(2+\alpha)}.$$
\end{thm}

\begin{rem}
In both Theorem \ref{T1} and \ref{T2} the first inequality of hypothesis H4 can be relaxed to
$$(1-\eps_0) \left [\left (x_n-\frac 1 \rho |x'|^2 \right )^+ \right ]^\alpha \le \det D^2 u \le (1+ \eps_0) \left (x_n + \frac 1 \rho |x'|^2 \right )^\alpha  \quad \quad \mbox{in $B_\rho \cap \Omega$}$$
or in other words we can replace $d_{\p \Om}$ by the distances to the exterior respectively interior tangent ball of radius $\rho$ at the origin. In fact in our proof we just use the inequality above instead of the first part of H4.
\end{rem}

Finally we also state a version of Theorem \ref{T2} in the case when the data is $C^2$ in a neighborhood of $0$.

\begin{thm}\label{T2.1} Let $\p \Omega \in C^2$ in $B_\rho$, and $u\in C(\ov \Om)$ convex such that
$$u(0)=0, \quad \nabla u(0)=0, \quad u=\ph(x') \quad \mbox{on $\p \Om \cap B_\rho$,}$$
with $$\ph \in C^2(B_\rho'), \quad \rho' I \le D_{x'}^2 \ph(0) \le \frac{1}{\rho'} \, I,$$ and $u \ge \rho'$ on $\p \Om \setminus B_\rho$.
Assume $$ \det D^2 u= g \, d_{\p \Om}^\alpha \quad \mbox{in} \quad \Om \cap B_\rho, \quad \quad \det D^2 u \le \frac{1}{\rho'} \quad \mbox{in} \quad \Om \setminus B_\rho $$
with $$g \in C^\beta(\ov \Om \cap B_\rho), \quad \quad \|g\|_{C^\beta} \le \frac{1}{\rho'}, \quad \rho'\le g(0) \le \frac{1}{\rho'},$$
for some $\beta>0$ small. Then $$u \in C^2 (\ov \Om \cap B_\delta)$$ with $\delta$ and the modulus of continuity of $D^2u$ depending on $n$, $\alpha$, $\beta$, $\rho$, $\rho'$ and the $C^2$ modulus of continuity of $\ph$ and $\p \Om$.
\end{thm}

\section{Preliminaries and rescaling}\label{s3}

In this section we use rescaling arguments and reduce the proof of Theorem \ref{T1} to the Proposition \ref{p2} below.

First we show that $|S_h|^2 d_h^\alpha \sim h^n$ where $d_h$ is the $e_n$ coordinate of the center of mass $x^*_h$ of $S_h$. We can think of $d_h$ also as a quantity that represents roughly the height of $S_h$ in the $x_n$ direction. In the next proposition we prove that after using a sliding $A_h$ depending on $h$ we may normalize $S_h$ such that it has its center of mass on the $x_n$-axis and the corresponding normalized function $\tilde u$ satisfies essentially the same hypotheses as $u$.

\begin{prop}\label{p1}
Assume $u$ satisfies the hypotheses of Theorem \ref{T1}. Then for all $h \le c(\rho,\rho', \eps_0)$ there exists a sliding along $x_n=0$ $$A_h=x-\tau_h x_n, \quad \quad \tau_h \cdot e_n=0, \quad |\tau_h| \le C(\rho,\rho', \eps_0)h^{-1/4},$$
such that the rescaled function $$\tilde u(A_h x)=u(x)$$ satisfies in $$\tilde S_h:=A_h S_h=\{ \tilde u <h\}$$ the following:

1) the center of mass $\tilde x^*_h$ of $\tilde S_h$ lies on the $x_n$ axis i.e. $\tilde x^*_h=d_h e_n.$

2) $$c_0 h^n \le |S_h|^2 d_h^ \alpha \le C_0 h^n,$$ with $c_0$, $C_0$ universal.
Also, after performing a rotation of the $x_1$,..,$x_{n-1}$ variables we can write
$$ \tilde x^*_h + c_0 D_h B_1 \subset \tilde S_h \subset C_0 D_h B_1,$$
where $$D_h:=diag(d_1,d_2,..,d_{n-1},d_n)$$ is a diagonal matrix that satisfies
\begin{equation}\label{dn}
\left(\prod_1^{n-1}d_i^2\right) \, d_n^{2+\alpha}=h^n.
\end{equation}

3) $$\tilde G_h:= \p \tilde S_h \cap \{ \tilde u < h\} \subset \p \tilde \Om_h$$ is a graph i.e $$\tilde G_h=(x', g_h(x')) \quad \quad \mbox{with} \quad g_h(x') \le \frac 2 \rho |x'|^2,$$
and the function $\tilde u$ satisfies on $\tilde G_h$ $$(1-2 \eps_0) \ph(x') \le \tilde u(x) \le (1+2 \eps_0) \ph(x').$$
Moreover $\tilde u$ satisfies in $\tilde S_h$
$$(1-2\eps_0) \left (x_n-\frac 4 \rho |x'|^2 \right )^\alpha \le \det D^2 \tilde u \le (1+ 2 \eps_0) \left (x_n + \frac 4 \rho |x'|^2 \right )^\alpha.$$

\end{prop}

For simplicity of notation in this section we denote shortly by $c'$, $C'$, $c_i'$, $C_i'$ various constants that depend on universal constants and $\rho$,$\rho'$ and $\eps_0$ (instead of $c(\rho,\rho', \eps_0)$ etc.) Also we use $c'$, $C'$ for constants that may change their value from line to line whenever there is no possibility of confusion.

First we construct an explicit barrier for $u$.

\begin{lem}\label{l1}
Let $$\bar w (r,y);= r^2 g(y r^{-\frac 32}) \quad \quad \mbox{with} \quad g(t)=(1-t^\gamma)^+, \quad t \ge 0,$$
for some $\gamma>0$ small depending only on $n$. Then the function
$$w_1(x',x_n):=c'\bar w (|x'|,C'x_n),$$
is a lower barrier for $u$ provided that $c'$ (small), $C'$ (large) are appropriate constants depending on $n$, $\mu$, $\rho$, $\rho'$.
\end{lem}

\begin{proof}
Let $t=y r^{-\frac 32}$. Using that
$$\frac {dt}{dr}=-\frac 32 t r^{-1}, \quad \quad \frac{dt}{dy}=r^{-\frac 32},$$ we compute in the set where $\bar w>0$ (hence $t \in (0,1)$):

\begin{align*}
 \bar w_{yy} &=r^{-1} g''=r^{-1} \gamma (1-\gamma) t^{\gamma-2},\\
 \bar w_r &=r(2g-\frac 32 t g')=r(2g+\frac 32 \gamma t ^\gamma),\\
 \bar w_{ry} &=r^{-\frac 12}(2g-\frac 32 t g')'=r^{-\frac 12} \gamma t^{\gamma-1}(-2+ \frac 32 \gamma),\\
  \bar w_{rr} &=(2g-\frac 32 t g')-\frac 32 t (2g-\frac 32 t g')'=2g+\frac 32 \gamma(3-\frac 32 \gamma)t^\gamma.
\end{align*}

 We find

 \begin{align*}
 \det D^2_{r,y} \bar w & \ge r^{-1} \gamma^2 t^{2 \gamma-2}\left[(1-\gamma) \frac 32 (3-\frac 32 \gamma)-(2-\frac 32 \gamma)^2  \right]\\
 & \ge c_0 r^{-1} t^{2 \gamma-2},
 \end{align*}
 and
 $$\frac{\bar w_r}{r}  \ge c_0 t^ \gamma,$$
thus
$$\det D_x^2 \bar w (|x'|,x_n) \ge c_1 |x'|^{-1} t^{n \gamma -2} \ge c_1 |x'|^{-1}.$$
Now we choose $c'=c(\rho,\rho')$ small such that
$$c'|x'|^2 \le \frac 14 \mu |x'|^2 \le u \quad \mbox{on $\p \Om \cap B_{\rho/2}$,}$$
$$c'|x'|^2 \le \rho' \quad \mbox{on $\p \Om \cap \{x_n \le \rho\}$}$$
 and then $C'$ large such that
 $$\det D^2 w_1 >1/ \rho' \quad \quad \mbox{on $B_{1/\rho} \cap \{ w_1>0\} $.}$$
 Since $u \ge w_1 $ on $\p (\Om \cap \{x_n \le \rho\})$ and $\det D^2 w_1 > \det D^2 u$ on the set where $w_1>0$ we find $u \ge w_1$ in $\Om \cap \{x_n \le \rho\}$.

\end{proof}

{\it Proof of Proposition \ref{p1}.}

Since $u \ge w$ with $w$ as in Lemma \ref{l1} we have $S_h \subset \{ w<h\}$ thus
$$S_h \subset \left \{ c'|x'|^2 (1-C'x_n |x'|^{-\frac 32}) < h  \right \},$$
or
\begin{equation}\label{1}
S_h \subset \left \{ |x'| \le C_1' h^ {1/2} \right \} \cup \left \{ x_n \ge c_1' |x'|^{\frac 32} \right \}.
\end{equation}
Let $x^*$ denote the center of mass of $S_h$ and define $d_h$ as
$$d_h=x^* \cdot e_n.$$
We claim that
\begin{equation}\label{2}
d_h \ge h^{3/4} \quad \quad \mbox{for all $h<c_2'$.}
\end{equation}
Otherwise we have
$$S_h \subset \{ |x'| \le C' h^{1/2} \} \cap \{ x_n \le C' h^{3/4} \},$$
and we compare $u$ with
$$w_2:=c' h \left [ \left ( \frac{|x'|}{h^{1/2}} \right)^2 + \left (\frac{x_n}{h^{3/4}} \right)^2 \right] + t x_n,$$ with $c'$ sufficiently small, and some $t>0$ arbitrarily small. In $S_h$ we have $w_2 \le h$ and
$$\det D^2 w_2 = c' h^{-1/2} \ge \det D^2 u,$$
and on $\p \Om \cap \p S_h$ we use $x_n \le C'|x'|^2$ and obtain
$$w_2 \le c'(|x'|^2 + C'|x'|^2 x_n h^{-1/2}) + t C' |x'|^2 \le \frac \mu 2 |x'|^2 \le u.$$
In conclusion $w_2 \le u$ which contradicts $\nabla u(0)=0$ and the claim \eqref{2} is proved.

Next we show that for all small $h$ we also have the following lower bound
\begin{equation}\label{3}
 d_h \le C h ^ \frac{1}{2+\alpha}.
 \end{equation}
Assume by contradiction that $d_h \ge C h^\frac{1}{2+\alpha}$ for some large $C$ universal. Since $S_h$ contains the set $\p \Om \cap B_{ch^{1/2}}$ and the point
$$x^*_h=({x^*_h}' , d_h) \quad \quad \mbox{with} \quad |{x^*_h}' | \le C' d_h^{2/3},$$
it contains also the convex set generated by them. It is straightforward to check that this convex set contains an ellipsoid $E$ of volume
$$|E|=c(n) (ch^{1/2})^{n-1} Ch^\frac{1}{2+\alpha},$$
with $c(n)$ a small constant depending only on $n$, such that
$$E\subset \{x_n -\frac 1 \rho |x'|^2  \ge h^\frac{1}{2+\alpha}\} \cap B_{\rho/2},$$
if $h$ is small. Now we compare $u$ with the quadratic polynomial $P$ that solves
$$ \det D^2 P = \frac \mu 2 (h^\frac{1}{2+\alpha})^\alpha \le \det D^2 u, \quad \quad P=h \ge u \quad \mbox{on} \quad \p E$$
hence $P \ge u \ge 0$. Writing this inequality at the center of $E$ we obtain
$$h^n \ge c(n) \, |E|^2 \, \det D^2 P,$$ and we reach a contradiction if $C$ is sufficiently large, hence \eqref{3} is proved.

 From \eqref{3} we see that $S_h \subset B_{\rho/2}$ for all small $h$, and the argument above shows in fact that
 \begin{equation}\label{4}
 |S_h|^2 \, d_h^ \alpha \le C_0 h^n,
 \end{equation}
 for all small $h$. Indeed, by John's lemma we can choose the ellipsoid $E$ centered at $x^*$ with
 $$E-x^* \subset \frac 1 4 (S_h - x^*), \quad \quad |E| \ge c(n) |S_h|,$$
 and $$E \subset \{ x_n -\frac 1 \rho |x'|^2  \ge d_h /2 \},$$
 and then we easily obtain \eqref{4} as before.

  Now we let
  $$\tilde x= A_h x:=x-\tau_h x_n, \quad \quad \tau_h:=\frac{{x_h^*}'}{x_h^* \cdot e_n},$$
  and
  $$ \tilde u(\tilde x)=\tilde u(A_hx)=u(x).$$
  From \eqref{1}, \eqref{2} we find
  \begin{equation}\label{5}
  |\tau_h| \le C'\frac{d_h^{2/3}}{d_h} \le C'd_h^{-1/3} \le C' h^{-1/4},
  \end{equation}
  and $\tilde x^*_h$ lies on the $x_n$ axis by construction.

  We have $\tilde x_n=x_n$ and if $x \in \p \Om \cap S_h \subset B_{Ch^{1/2}}$ then
  $$|x-\tilde x|=|\tau_h x_n| \le C' h^{-1/4} |x'|^2 \le C'h^{1/4}|x'|.$$
  This easily implies that $\tilde G_h$ defined in Proposition \ref{p1} belongs to the graph of a function $g_h$ that satisfies $|g_h(x')| \le (2 / \rho) |x'|^2$. Since
  $$|\ph(x')-\ph(\tilde x')| \le C|x'| |x'-\tilde x'| \le C' h^{1/4} |x'|^2 \le \frac{\eps_0}{2} \varphi(x'),$$
on $\tilde G_h$ we have
$$(1-2\eps_0) \ph(\tilde x ') \le \tilde u(\tilde x) \le (1+2\eps_0) \ph(\tilde x').$$

Also if $x \in S_h$ then (see \eqref{5}, \eqref{3})
$$|x'|^2 \le 2 |\tilde x'|^2 + 2 |\tau_h|^2 x_n^2 \le 2|x'|^2 + C'd_h^{-2/3} d_h x_n \le 2|\tilde x'|^2 + \frac{\eps_0 \rho}{2} x_n  $$
thus
$$x_n+ \frac 1 \rho |x'|^2 \le (1+ \frac{\eps_0}{2})(\tilde x_n + \frac 4 \rho |\tilde x'|^2),$$
$$x_n- \frac 1 \rho |x'|^2 \ge (1 - \frac{\eps_0}{2})(\tilde x_n - \frac 4 \rho |\tilde x'|^2),$$
which imply the desired inequalities for $\det D^2 \tilde u$.

It remains to show part 2) of Proposition \ref{p1}. After a rotation of the first $n-1$ coordinates we may assume that $\tilde S_h \cap \{x_n=d_h\}$ is equivalent to an ellipsoid of axes $d_1\le d_2 \le \cdots \le d_{n-1}$ i.e.
$$\left \{ \sum_1^{n-1} (\frac {x_i}{d_i})^2 \le 1 \right \} \cap \{ x_n = d_h\} \subset \tilde S_h \cap \{x_n =d_n \} \subset \left \{ \sum_1^{n-1} (\frac {x_i}{d_i})^2 \le C(n) \right \},$$
with $C(n)$ a constant depending only on $n$. We find
$$S_h \subset \left \{ \sum_1^{n-1} (\frac {x_i}{d_i})^2 \le C(n) \right \} \cap \{ 0 \le x_n \le C(n) d_h \},$$
and also since $\tilde u \le c |x'|^2$ on $\tilde G_h$ we see that
\begin{equation}\label{6}
d_i \ge c_3 h^{1/2}.
\end{equation}
We claim that
\begin{equation}\label{7}
d_h^{2+\alpha}\prod_1^{n-1} d_i^2 \ge c_4 h^n.
\end{equation}
Otherwise, similarly as before we consider
$$w_3:=ch\left[\sum_1^{n-1} (\frac{x_i}{d_i})^2 + (\frac{x_n}{d_h})^2  \right] + t x_n,$$
with $c$ small, and obtain (provided that $c_4$ is chosen sufficiently small)
$$\det D^2 w_3 \ge c^n h^n (d_h^2 \Pi d_i^2)^{-1} \ge C d_h^ \alpha \ge \det D^2 \tilde u,$$
$$w_3 \le h =\tilde u \quad \mbox{on} \quad \p \tilde S_h \setminus \tilde G_h,$$
and moreover on $\tilde G_h$ we use \eqref{6} and obtain
$$w_3 \le c|x'|^2 + C h \frac{x_n}{d_h} + t x_n\le \frac{\mu}{4}|x'|^2 \le \tilde u.$$
This implies $\tilde u \ge w_3$ in $\tilde S_h$ and we contradict that $\nabla \tilde u(0)=0$, hence \eqref{7} is proved.

Now we define $d_n$ from $d_1$, .., $d_{n-1}$ by the equality \eqref{dn}, and \eqref{4}, \eqref{7} give
\begin{equation}\label{7.1}
cd_n \le d_h \le C d_n
\end{equation}
which proves part 2).

\qed

\begin{rem}\label{r1}
The set $\tilde S_h \cap \{x_n=d_h\}$ is just a translation of $S_h \cap \{x_n=d_h\}$, hence $d_1$, $d_2$,..,$d_{n-1}$ represent the length of the axes of an ellipsoid which is equivalent to $S_h \cap \{x_n=x^*_h \cdot e_n\}$.
\end{rem}

\begin{rem}
We can prove \eqref{7} without using the upper bound on $\ph(x')$. Precisely, if we assume that $\ph$ satisfies
$$\mu \mathcal N \le D_{x'}^2 \ph \le \mu^{-1} \mathcal N,$$
with $$\mathcal N =diag(a_1^2,\ldots, a_{n-1}^2), \quad \quad a_i \ge 1,$$ then \eqref{7} still holds. Indeed, now we have $d_i \ge c_3 h^{1/2}/a_i$ instead of \eqref{6} and then on $\tilde G_h$ we still satisfy
$$w_3 \le c a_i^2x_i^2 + C x_n h/d_h + t x_n \le \ph(x') \le \tilde u.$$
\end{rem}

We mention that in the beginning of the proof of Proposition \ref{p1} we obtained a pointwise $C^{1,1/3}$ estimate for solutions that grow quadratically away from their tangent plane and have bounded Monge-Ampere measure. We state this result below although it will not be used in the proof of Theorem \ref{T1}.

\begin{prop}\label{p0}
Assume $\Om$, $u$ satisfy hypotheses H1, H2 of Section \ref{s2}, and
$$\rho |x'|^2 \le u(x) \le \frac {1}{\rho} |x'|^2 \quad \mbox{on} \quad \p \Om, \quad \quad \det D^2u \le \frac {1}{\rho} \quad \mbox{in} \quad \Om.$$ Then
$$u(x) \le C'|x|^\frac 43  \quad \quad \mbox{in} \quad \Om \cap B_{c'}$$
with $C'$, $c'$ constants depending on $n$ and $\rho$.
\end{prop}

\begin{proof}
The section $S_h$ and its center of mass $x_h^*$ satisfy \eqref{1} and \eqref{2} since we only used the upper bound on $\det D^2 u$ and the quadratic bound by below for $u$ on $\p \Om$. From this we obtain that the convex hull generated by $x_h^*$ and $\p \Om \cap B_{c'h^{1/2}}$, which is included in $S_h$, contains $\ov \Om \cap B_{c'_1h^{3/4}}$ for some small $c_1'$, which proves the proposition.

\end{proof}

In order to prove Theorem \ref{T1} we need to show that the quantities $d_i$ are bounded by above by $C h^{1/2}$ for some $C$ universal. Precisely we prove the following lemma which will be completed in Section \ref{s6}.

\begin{lem}\label{l2}
Assume $u$ satisfies the hypotheses of Proposition \ref{p1} for some $\eps_0$ sufficiently small, universal. Then for all $h \le c(\rho,\rho')$ we have
$$\max_{1\le i \le n-1} \, \,  d_i \le C h^{1/2},$$
for some $C$ universal, with $d_i$ defined as in Proposition \ref{p1}.
\end{lem}

{\it Lemma \ref{l2} implies Theorem \ref{T1}}

From Lemma \ref{l2} and \eqref{6}, \eqref{dn} we find ($i \ne n$)
$$ch^{1/2} \le d_i \le Ch^{1/2}, \quad ch^\frac{1}{2+\alpha} \le d_n \le C h^\frac{1}{2+\alpha},$$
hence, by Proposition \ref{p1},
\begin{equation}\label{8}
\tilde x_h^* + c F_h B_1 \subset A_h S_h \subset C F_h B_1,
\end{equation}
with $$F_hx:=(h^\frac 12x',h^\frac{1}{2+\alpha}x_n).$$
Since $ \p \Om_h \cap B_{ch^{1/2}} \subset \tilde G_h \subset \tilde S_h = A_h S_h$ we see from the inclusion above that also
$$c F_h B_1 \cap A_h \ov \Om \subset A_h S_h \subset C F_h B_1.$$

Using in \eqref{8} that $S_{h/2} \subset S_h$ we find
$$F_h^{-1} A_h A_{h/2}^{-1} F_{h/2} B_1 \subset C B_1$$ which gives
$$|\tau_h - \tau_{h/2}| \le C_1 h^{\frac12-\frac{1}{2+\alpha}},$$
for all $h \le c(\rho,\rho')$. If we denote by $h_k=2^{-k}$ then, since $\alpha>0$, we obtain $\tau_{h_k} \to \tau_0$ and
$$|\tau_h-\tau_0| \le C_2 h^{\frac12-\frac{1}{2+\alpha}}, \quad \quad \mbox{for all $h=h_k \le c(\rho,\rho')$.}$$
This inequality implies
$$cB_1 \subset F_h^{-1} A_h A_0^{-1} F_h \, \, B_1 \subset C \, B_1,$$
hence we can replace $A_h$ with $A_0$ in the second inclusion above and obtain
$$k F_h B_1 \cap A_0 \ov \Om \subset A_0 S_h \subset k^{-1} F_h B_1,$$
for some small $k$ universal.

\qed

{\bf Normalized solutions.}

Next we ``normalize" $\tilde u$ in $\tilde S_h$ (or we may think we normalize $u$ in $S_h$) back to size 1 in such a way that it solves a similar equation. Precisely we define
\begin{equation}\label{v}
v(x):=\frac{1}{h} \tilde u(D_hx)=\frac {1}{h} \tilde u(d_1 x_1,\ldots,d_n x_n)
\end{equation}
with
$D_h$, $d_1, \ldots , d_n$ defined in Proposition \ref{p1}. Then $v$ is a continuous convex function function in $\ov \Om_v$ with $\Om_v:=D_h^{-1} \tilde \Om$ and
\begin{equation}\label{v1}
v(0)=0, \quad \quad v\ge 0, \quad \quad \nabla v(0)=0 \quad \quad \mbox{(in the sense of H2).}
\end{equation}
The section $S_1(v):=\{ v<1\}$ satisfies $S_1(v)=D_h^{-1} \tilde S_h$ thus
\begin{equation}\label{v2}
x^* + c B_1 \subset S_1(v) \subset C B_1, \quad \quad \mbox{for some point $x^*$.}
\end{equation}
We compute
$$ \det D^2 v(x) = h^{-n} (\det D_h)^2 \det D^2 \tilde u(D_hx)=d_n^{-\alpha} \det D^2 \tilde u (D_hx).$$
From \eqref{1}, \eqref{7.1} we know that for $i <n$ we have $d_i \le C' d_n^{2/3}$ hence
$$\left |\frac 4 \rho \sum_1^{n-1}(d_i x_i)^2 \right | \le C' d_n^{4/3} |x'|^2 \le \eps_0 d_n |x'|^2,$$
if $h <c'$. Using this inequality in Proposition \ref{p1} part 3) we obtain
$$(1-2\eps_0)[(x_n-\eps_0|x'|^2)^+]^\alpha \le d_n^{-\alpha} \det D^2 \tilde u(D_hx) \le (1+ 2 \eps_0)(x_n+\eps_0|x'|^2)^\alpha,$$
hence
\begin{equation}\label{v3}
 (1-2\eps_0)[(x_n-\eps_0 |x'|^2)^+]^\alpha \le \det D^2 v \le (1+ 2 \eps_0)(x_n+\eps_0|x'|^2)^\alpha  \quad \quad \mbox{in $S_1(v)$.}
\end{equation}

If we denote by $G_v$ the closed set $G_v:=\p \Om_v \cap \p S_1(v)$ we have that $G_v$ is the graph of a convex function $(x',g_v(x'))$ with
$$d_n g_v \le \frac2 \rho \sum d_i^2 x_i^2 \le \eps_0 d_n |x'|^2,$$
hence
\begin{equation}\label{v4}
0 \le g_v \le \eps_0 |x'|^2.
\end{equation}
We have $v=1$ on $\p S_1(v) \setminus G_v$, and on $G_v$ the function $v$ satisfies
\begin{equation}\label{v5}
(1-2\eps_0) \ph_v(x') \le v \le (1+2\eps_0) \ph_v(x'),
\end{equation}
with
$$\ph_v(x'):=\frac 1 h \ph(d_1x_1,\ldots, d_{n-1}x_{n-1}).$$
Notice that
$$\mu^{-1} \mathcal N \ge D^2_{x'} \ph_v \ge \mu \mathcal N,$$
with (see \eqref{6}) $$\mathcal N=diag(a_1^2,a_2^2,\ldots,a_{n-1}^2), \quad \quad a_i:=\frac{d_i}{h^{1/2}} \ge c.$$
We collect the properties \eqref{v1}-\eqref{v5} for $v$ into a formal definition below.

{\bf The class $\mathcal D_{\bar \mu}^\sigma$.}

Let $\bar \mu$, $\sigma$ be positive (small) fixed constants, and let $\bar \mu \le a_1 \le \cdots \le a_{n-1}$ be real numbers.

We say that $$v \in \mathcal D_{\bar \mu}^\sigma(a_1,\ldots,a_{n-1})$$ if $v$ is a continuous convex function defined on a convex set $\ov \Om$ such that

1)
$$0 \in \p \Om, \quad B_{\bar\mu}(x_0) \subset \Omega \subset B_{1/\bar \mu}^+ \quad \mbox{for some $x_0$,}$$
$$1 \ge v\ge 0, \quad v(0)=0 \quad \nabla u(0)=0,$$

2) in the interior of $\Om$ the function $v$ satisfies:
$$(1-\sigma)[(x_n-\sigma |x'|^2)^+]^\alpha \le \det D^2 v \le (1+ \sigma)(x_n+\sigma|x'|^2)^\alpha,$$

3) on $\p \Om$ the function $v$ satisfies:

\noindent there exists a closed set $G \subset \p \Om$ which is a graph $(x',g(x'))$ with
$$g(x') \le \sigma |x'|^2,$$ such that
$$v=1 \quad \mbox{on $\p \Om \setminus G$,}$$
and
$$(1-\sigma) \ph_v(x') \le v \le (1+\sigma) \ph_v(x') \quad \quad \mbox{on $G$}$$
for some function $\ph_v$ such that $$\bar \mu^{-1} \mathcal N \ge D^2_{x'} \ph_v \ge \bar \mu \mathcal N, \quad \quad \mbox{with} \quad \mathcal N=diag(a_1^2,a_2^2,\ldots,a_{n-1}^2).$$

\

In view of \eqref{v1}-\eqref{v5} and the definition above we may rephrase Proposition \ref{p1} as follows.

\begin{lem}\label{l3}
If $u$ satisfies the hypotheses of Proposition \ref{p1} and $v$ is the normalized solution of $u$ in $S_h$ given by \eqref{v}, and $h \le h_0:=c(\rho,\rho',\eps_0)$, then
$$v \in \mathcal D_{\bar \mu}^{2 \eps_0}(a_1,a_2,..,a_{n-1}),$$
for some $\bar \mu$ universal (depending on $n$, $\alpha$, $\mu$) and with $a_i = d_i h^{-\frac 12}$.
\end{lem}

\

{\bf Definition of $\mathcal S_h'(u)$.}

Given a section $S_h(u)$ at the origin for some convex function $u$, we define the set $\mathcal S'_h(u) \subset \R^{n-1}$ (and call it {\it normalized diameter} of $S_h(u)$) as
$$x' \in \mathcal S'_h(u) \Leftrightarrow x^*_h + h^\frac 12(x',0) \in S_h(u),$$
where $x^*_h$ denotes the center of mass of $S_h(u)$. In other words $\mathcal S_h'$ is obtained by intersecting $S_h$ with the $n-1$ dimensional plane generated by $e_1$, ..$e_{n-1}$ passing through its center of mass, and then we perform a $h^{- 1/2}$ dilation.

From the definition we see that if $\tilde u(Ax)=u(x)$ with $A$ a sliding along $\{x_n=0\}$ then $\mathcal S_h'(\tilde u)=\mathcal S_h(u)$. If $u$ satisfies the hypotheses of Proposition \ref{p1} then, by the definition of $d_i$ (see Remark \ref{r1}), we have that $\mathcal S_h'(u)$ is equivalent to the $n-1$ dimensional ellipsoid $E_h$ of axes $a_i=d_i h^{-1/2}$, $i<n$ i.e.
\begin{equation}\label{12}
E_h \subset \mathcal S_h'(u) \subset C(n) E_h.
\end{equation}
Thus Lemma \ref{l2} is equivalent to showing that $\mathcal S_h'(u)$ is included in a fixed ball of universal radius for all $h$ small.

Next we check the relation between $\mathcal S'_t(v)$ and $\mathcal S'_{th}(u)$ if $v$ is the normalized solution for $u$ in $S_h$. Since
$$v=\frac 1 h \tilde u(D_h x)$$ we have $S_t(v)=D_h^{-1} S_{th}(\tilde u)$ hence
\begin{equation}\label{11}
\mathcal S_t'(v)= h^\frac 1 2 {D_h'}^{-1} \mathcal S'_{th}(\tilde u) = h^ \frac 12 {D_h'}^{-1} \mathcal S'_{th}(u),
\end{equation}
where $D_h'=diag(d_1,..,d_{n-1})$ represents the restriction of $D_h$ to the first $n-1$ variables.

In order to prove Lemma \ref{l2} and therefore Theorem \ref{T1} it suffices to prove the next proposition which provides bounds for the sets $\mathcal S_t'(v)$ for general functions $v \in \mathcal D_\sigma^{\bar \mu}$.

\begin{prop}\label{p2}
Let $\bar \mu$ small, $M$ large be fixed. There exist positive constants $\delta$, $\bar c$ small, depending only on $\bar \mu$, $n$, $\alpha$, $M$ such that if $$v \in \mathcal D_\delta^{\bar \mu}(a_1,..,a_{n-1}), \quad \quad \mbox{and} \quad a_{k+1} \ge \delta^{-1},$$
for some $0 \le k \le n-2$, then
$$\mathcal S_t'(v) \subset \{|(x_{k+1},..,x_{n-1})| \le \frac 1 M \},$$
for some $t \in [\bar c, 1]$.
\end{prop}

\begin{rem}
Since $S_1(v) \subset B_{1/\bar \mu}$ we always have the inclusion
\begin{equation}\label{10}
\mathcal S_t'(v) \subset t^{-\frac 1 2} B_{2/\bar \mu}'.
\end{equation}
Proposition \ref{p2} states roughly that if the boundary data of $v$ grows sufficiently fast in the $(x_k,..,x_{n-1})$ variables then the normalized diameter $\mathcal S_t'(v)$ projects into an arbitrarily ``small" set in these variables.
\end{rem}

The proof of Proposition \ref{p2} will be completed in the next three sections. We conclude this section by showing that Lemma \ref{l2} follows from Proposition \ref{p2}.

\

\begin{lem}\label{l3.1}
 Proposition \ref{p2} implies Lemma \ref{l2}
\end{lem}

\begin{proof} We apply Proposition \ref{p2} for $\bar \mu$ as in Lemma \ref{l3} and for $M:=4 \sqrt n$, hence the constants $\delta$, $\bar c$ above become universal constants. We also choose $\eps_0=\delta /2$ so that Proposition \ref{p2} applies for all normalized functions of $u$ in $S_h$ with $ h\le h_0$, with $h_0=c(\rho,\rho')$.

Denote by $d_i(h)$ and $a_i(h)$ the quantities $d_i$ and $a_i=d_i h^{-1/2}$ (for $i<n$) corresponding to the section $S_h$. We show that for any $h \le h_0$ we have
\begin{equation}\label{9}
\max_i a_i(h) \ge \bar C \quad \quad \Rightarrow \quad \quad \max_i a_i(th) \le \frac 12 \max_i a_i(h),
\end{equation}
for some $t \in [\bar c, 1]$, and with $\bar C$ universal.

Since $S_{h_0} \subset B_{1/\rho}$ we find $d_i(h_0) \le \rho^{-1}$ hence
$$\max a_i(h_0) \le C_0':=\rho^{-1} h_0^{- \frac 1 2}.$$
Now property \eqref{9} implies that $\max a_i(h)$ is bounded above by a universal constant for all $h \le c_1'$, thus Lemma \ref{l2} holds.

In order to prove \eqref{9} let $v$ denote the normalized function for $u$ in $S_h$ and assume that $a_{k+1}(h)$ is the first $a_i(h)$ greater than $\delta^{-1}$ i.e.
$$a_1 \le \cdots \le a_{k} \le \delta^{-1} \le a_{k+1}  \le \cdots \le a_{n-1}.$$
Since $v \in \mathcal D_\delta^{\bar \mu} (a_1,..,a_{n-1})$, by Proposition \ref{p2} we have (see \eqref{10})
$$\mathcal S_t'(v) \subset \{|(x_1,..,x_k)| \le C_1\} \times \{|(x_{k+1},..,x_{n-1})| \le \frac 1 M\},$$
for some $C_1$ universal with $$C_1:=2\bar c^{-\frac 12}/\bar \mu \ge 2t^{-\frac 12}/ \bar \mu.$$
From \eqref{11} $$\mathcal S'_{th}(u)=h^{-\frac 12} D_h' \mathcal S_t'(v)=diag(a_1,..,a_{n-1}) \mathcal S_t'(v),$$ and we obtain
$$\mathcal S'_{th}(u) \subset \prod_{i=1}^{k}\{|x_i|\le C_1 a_i\} \times \prod_{i=k+1}^{n-1}\{|x_i|\le \frac{a_i}{M}\}.$$
For $i \le k$ we have
$$C_1 a_i \le C_1 \delta^{-1} := \frac {\bar C} { M} \le \frac{\max a_i} {M},$$
 and we find
$$\mathcal S'_{th}(u) \subset \frac 14 \max a_i(h) B'_1,$$
which gives (see \eqref{12}) $$\max a_i(th) \le \frac 12  \max a_i(h).$$
\end{proof}

\section{Compactness and the class $\mathcal D_0^\mu$}\label{s4}

In this section we use compactness arguments and reduce Proposition \ref{p2} to the Theorem \ref{T3} below.

We prove Proposition \ref{p2} by compactness by letting $\sigma \to 0$ and $a_{k+1} \to \infty$.

First we remark that if we have a sequence of functions $v_m$ in $D_{\sigma_m}^\mu$ with $\sigma_m \to 0$ then we can extract a subsequence $v_{m_l}$ that converges to a limiting convex function $v$. Here, and throughout this paper, the convergence of convex functions (defined on possibly different domains) means that their supergraphs converge in the Hausdorff distance (in $\R^{n+1}$) to the supergraph of the limit function. The Monge-Ampere measure of the limit function $v$ is given by $x_n^\alpha$, however $v$ may have discontinuities at the boundary. Before we introduce the class $\mathcal D_0^\mu$ of such limiting solutions, we recall some definitions of boundary values for convex functions defined in convex domains (see \cite{S1}).

\begin{defn}
Let $u:\ov \Om \to \R$ convex, and $\ph: \p \Om \to \R$ be two bounded semicontinuous functions i.e. their upper graph
$$\{x_{n+1} \ge u(x)\} \subset \ov \Om \times \R, \quad \quad \{x_{n+1} \ge \ph(x) \} \subset \p \Om \times \R,$$ are closed sets. We say that $$u=\ph \quad \mbox{on} \quad \p \Om$$
if $u|_{\p \Om}=\ph^*$ where $\ph^*$ represents the convex envelope of $\ph$. In other words $u=\ph$ on $\p \Om$ means that, when we restrict to the cylinder $\p \Om \times \R$, the upper graph of $u$ coincides with the convex envelope of the upper graph of $\ph$.
\end{defn}

An example of function $\ph$ is of course $u|_{\p \Om}$, the restriction of $u$ to $\p \Om$, and when $\Om$ is strictly convex this is the only possible choice. On the other hand, on some flat part of the boundary $\p \Om$ there are many choices of functions $\ph \ge u$ since we only require $\ph^*=u$. The advantage of the definition above is that the maximum principle still holds and the boundary data behaves well when taking limits. Precisely we have (see Proposition 2.2 and Theorem 2.7 in \cite{S1}):

{\it Maximum Principle:} Assume
$$u=\ph, \quad v=\psi,  \quad \ph \le \psi \quad \mbox{on $\p \Omega$},$$
$$\det D^2 u \ge f \ge \det D^2 v \quad \mbox{in $\Om$.} $$
Then $u \le v$.

\

{\it Closedness under limits:} Assume $$\det D^2 u_k=f_k, \quad \quad u_k=\ph_k \quad \mbox{on} \quad \p \Om_k,$$  and
$$u_k \to u, \quad \ph_k \to \ph, \quad f_k \to f.$$ Then
$$\det D^2 u=f, \quad \quad\mbox{and} \quad u=\ph \quad \mbox{on} \quad \p \Om.$$

\

By $u_k \to u$, $\ph_k \to \ph$ above we understand that the corresponding upper graphs converge in the Hausdorff distance and $f_k \to f$ means that $f_k$ converges uniformly on compact sets to $f$.

We also use the following property of boundary values as defined above: if $u=\ph$ on $\p \Om$ then the restriction of $u$ to the set $\{ u \le h\}$
satisfies
$$\mbox{ $u=\ph$ on  $\{\ph \le h\}$ and $u=h$ on the rest of $\p \{u<h\}$}.$$

 Next we introduce the class $\mathcal D_0^\mu$. By abuse of notation we denote its elements still by $u$ and they can be viewed as limits of normalized solutions of the functions $u$ from Section 2.

 {\bf The class $\mathcal D_0^\mu$.}

 Let $\mu>0$ be fixed, and let $\mu \le a_1 \le \cdots \le a_k$ be $k$ real numbers, $0 \le k \le n-1$. We say that the convex function $u$ defined in the convex set $\ov \Om$ belongs to the class
 $$u \in \mathcal D_0 ^\mu(a_1,..,a_k, \infty,..,\infty)$$
 if the following hold:

 1) $$0 \in \p \Om, \quad B_\mu(x^*) \subset \Om \subset B_{1/\mu}^+ \quad \mbox{for some $x^*$},$$
 $$u \ge 0, \quad u(0)=0, \quad \nabla u(0)=0,$$

 2) $$\det D^2 u=x_n^\alpha \quad \mbox{in $\Om$,}$$

 3) $$ \mbox{$u=\ph$ on $\p \Om$ with} \quad \quad \ph:=\left \{
 \begin{array}{l}
   \psi_u \quad \quad \mbox{on $G \subset \p \Om$,}\\
 1 \quad \quad \quad \mbox{on $\p \Om \setminus G$,}
 \end{array}
 \right.
 $$
 where $\psi_u(x_1,..x_k)$ is a nonnegative convex function of $k$ variables satisfying
 $$\mu^{-1} \mathcal N_k \ge D^2 \psi_u \ge \mu \mathcal N_k, \quad \quad \quad \mathcal N_k:=diag(a_1^2,..,a_k^2),$$
 and $G$ represents the $k$ dimensional set (in $\R^n$) where $\psi_u \le 1$, i.e.
 $$G:=\{x \in \R^n| \quad \psi_u(x_1,..,x_k) \le 1, \quad x_i=0  \quad \mbox{if} \quad i>k \}.$$

We easily obtain the following lemma

\begin{lem}[Compactness]\label{l4}
Assume $v_m \in \mathcal D_{\sigma_m}^\mu(a_1^m,..,a_{n-1}^m)$ is a sequence of functions with
$$\sigma_m \to 0, \quad a_{k+1}^m \to \infty.$$
Then we can extract a convergent subsequence to a function $u$ with
$$u \in \mathcal D_0^\mu(a_1,..,a_l,\infty,.,\infty)$$
for some $0 \le l \le k$.
\end{lem}

\begin{proof}
All the properties for $u$, except $\nabla u(0)=0$, follow from the closedness under limits property above. In order to show that $\nabla u(0)=0$ we remark that if $v \in D_\sigma^\mu$ (with $\sigma \in [0,1/2)$) then we can obtain from the proof of Proposition \ref{p1} that the center of mass $x^*_h(v)$ of $S_h(v)$ satisfies
\begin{equation}\label{13}
x_h^*(v)\cdot e_n \ge  h^\frac 34 \quad \quad \mbox{if $h \le c$,}
\end{equation}
where $c$ depends only on $n$, $\alpha$, $\mu$. Indeed, we bound $v$ by below using the same barriers $w_1$ and $w_2$ (with constants depending only on $n$, $\alpha$, $\mu$) and obtain the estimates \eqref{1}, \eqref{2}. We can do this since we only need the inequality $v \ge c|x'|^2$ on the part of the boundary where $\{v<1\}$ which is clearly satisfied by all $v \in \mathcal D_\sigma^\mu$.

Since property \eqref{13} is preserved after taking limits we see that $u$ satisfies it as well, and this easily implies that $\nabla u(0)=0$  since otherwise $S_h(u) \subset \{ x_n \le O(h) \}$ and we contradict \eqref{13}.

\end{proof}

\begin{rem}\label{r3}
In the proof above we allow $\sigma_m=0$ and $a_i=\infty$ for some $i$, therefore the compactness holds for the class $\mathcal D_0^\mu$ as well.
\end{rem}

Using the compactness lemma above we see that in order to prove Proposition \ref{p2} it suffices to prove the following version for the class $\mathcal D_0^\mu$.

\begin{prop}\label{p3}
Let $\mu>0$ small, $M>0$ large be fixed, and assume $$u \in \mathcal D_0^\mu(a_1,..a_k,\infty,..,\infty)$$ for some $0 \le k \le n-2.$ There exists $\bar c(k,M) >0$ depending only on $n$, $\alpha$, $\mu$, $M$, $k$ such that
$$\mathcal S'_t(u) \subset \{ |(x_{k+1},..,x_{n-1})| \le \frac 1 M \},$$
for some $t \in [\bar c(k,M),1]$.
 \end{prop}

We will prove Proposition \ref{p3} by induction on $k$, and this is the reason why we require the dependence of $\bar c$ on $k$. Clearly at the end, the constant $\bar c(M)$ which is the minimum of all $c(k,M)$ above can be taken independent of $k$.

 \

 {\it Proposition \ref{p3} implies Proposition \ref{p2}.}

 We show that Proposition \ref{p2} holds with the constant $$\bar c:= \bar c(2M) =\min_k \bar c(k,2M)$$ and for some $\delta>0$ small. Otherwise there exists a sequence of $\delta_m \to 0$ and corresponding functions $v_m \in D_{\delta_m}^\mu$, $a_k^m \to \infty$ for which the conclusion of Proposition \ref{p2} does not hold. By Lemma \ref{l4} we can extract a convergent subsequence to a function $$u \in \mathcal D_0^\mu(a_1,..,a_l,\infty,..,\infty) \quad \quad \mbox{for some $0 \le l \le k$}.$$ From Proposition \ref{p3} there is $t \in [\bar c,1]$ such that
 $$\mathcal S'_t (u) \subset \{|(x_{l+1},..,x_{n-1})| \le \frac{1}{2M} \} \subset \{|(x_{k+1},..,x_{n-1})| \le \frac{1}{2M}\},$$
and therefore the conclusion is satisfied for $v_m$ for all large $m$, contradiction.

\qed

The key step in proving Proposition \ref{p3} consists in proving the following estimates for the class $\mathcal D_0^\mu (1,1,.,1,\infty,..,\infty)$.


\begin{thm}\label{T3}
If $$u \in \mathcal D_0^\mu(\underbrace{1,...,1}_{k \, \, times},\infty..,\infty)$$ for some $0\le k \le n-2$ then
$$\mathcal S'_h(u) \subset \{ |(x_{k+1},..,x_{n-1})| \le C h^\beta  \},\quad \quad \quad \beta:=\frac{1}{2(n+1-k+\alpha)}>0,$$
with $C$ large depending only on $n$, $\alpha$, $\mu$ and $k$.
\end{thm}

Theorem \ref{T3} holds for $\alpha=0$ as well. Its proof will be completed in Section \ref{s6} by induction on $k$ and we will see that it applies also for $\alpha=0$.

\begin{lem}\label{l4.1}
Theorem \ref{T3} implies Proposition \ref{p3}.
\end{lem}

\begin{proof}
We prove Proposition \ref{p3} by induction.

{\it Case $k=0$:} We apply Theorem \ref{T3} for $k=0$ and obtain that
 $$\mathcal S'_t(u) \subset \{ |(x_1,..,x_{n-1})| \le C t^\beta \le \frac 1 M \},$$ by choosing $t$ small depending on $M$, $\mu$, $n$, $\alpha$.

{\it Case $k-1 \Rightarrow k$.} We assume Proposition \ref{p3} holds for $k-1$ with $1 \le k \le n-2$ and we prove it holds also for $k$. By compactness (see Remark \ref{r3}) we know that the following property holds

{\it Property $P(k-1)$:}

There exists $C_0:=C_0(M,\mu,k,n,\alpha)$ such that if
$$u \in D_0^\mu(a_1,..,a_{n-1}), \quad \quad \mbox{with $a_k \ge C_0$},$$
for some $a_i \in [\mu,\infty) \cup \{\infty \}$, then
$$\mathcal S_t'(u) \subset \{|(x_k,..,x_{n-1})| \le \frac 1 M \},$$
for some $t\in [c_k,1]$ with $c_k$ depending on the parameters above.

\

Thus when $a_{k} \ge C_0$ the conclusion for $k$ i.e.
$$\mathcal S_t'(u) \subset \{|(x_{k+1},..,x_{n-1})| \le \frac 1 M \}$$
is already satisfied from the property $P(k-1)$.
It remains to prove the statement only when $u\in \mathcal D_0^\mu(a_1,..,a_k,\infty,..,\infty)$ and $a_k \le C_0$. In this case we can write
$$u \in \mathcal D_0^{\tilde \mu} (\underbrace{1,...,1}_{k \, \, times},\infty..,\infty)$$
for some $\tilde \mu$ small depending on $M$, $\mu$, $n$, $k$, $\alpha$.
 Now we can apply Theorem \ref{T3} and find that
 $$S_t'(u)\subset \{|(x_{k+1},..,x_{n-1}| \le C(\tilde \mu) t^\beta  \le \frac 1 M\},$$ if we choose $t$ small enough depending on $M$, $\mu$, $n$, $k$, $\alpha$.

\end{proof}

\begin{rem}\label{r2}
In the proof above we showed that if Theorem \ref{T3} holds for all $l\le k$, for some $k$ satisfying $0 \le k \le n-2$, then Proposition \ref{p2} holds for all $l\le k$ as well.
\end{rem}

We conclude this section with some results for normalized solutions of $u \in \mathcal D_0^\mu$ in $S_h(u)$, that are versions of Proposition \ref{p1} for the class $\mathcal D_0^\mu.$ Below we think of $\mu$, $\alpha$, $n$ as being fixed constants, and we refer to other positive constants depending only on $\mu$, $\alpha$ and $n$ as universal constants.

\begin{lem}\label{l6}
Assume $$u \in \mathcal D_0^\mu (\infty,...,\infty).$$
For each $h \in (0,1]$, after a rotation (relabeling) of the $x'$ coordinates, there exist a sliding $A_h$ along $x_n=0$, and a diagonal matrix $D_h$
$$D_h=diag (d_1,d_2,..,d_n), \quad \quad \mbox{with} \quad \left(\prod_{i=1}^n  d_i^2 \right ) \, \, d_n^\alpha=h^n,$$
such that the normalized solution
$$u_h(x):=\frac 1 h u(A_h D_h x) \quad \quad \mbox{satisfies} \quad u_h|_{S_1(u_h)} \in \mathcal D_0^{\bar \mu}(\infty,..,\infty),$$
with $\bar \mu>0$ universal.
Moreover if $x_h^*$ denotes the center of mass of $S_h(u)$ then, $$c d_n \le x_h^* \cdot e_n \le C d_n \quad \mbox{and} \quad ch^{3/4} \le d_n \le C h^\delta,$$
with $c$, $C$, $\delta$ universal.
\end{lem}

\begin{proof}
This is a simplified version of Proposition \ref{p1} since the behavior of $\det D^2 u_h$ and the boundary data of $u_h$ is left invariant under composition of the affine transformations above.

As in the proof of Proposition \ref{p1}, we choose $d_1$,...,$d_{n-1}$ as being the lengths of the axes of the ellipsoid which is equivalent to $S_h \cap \{x_n=x_h^*\cdot e_n\}$. After a rotation we may assume that its axes are parallel to the coordinate axes. We choose $d_n$ in terms of $d_1,\ldots,d_{n-1}$ so that it satisfies the above identity for $D_h$. We let $A_h$ so that $x^*$, the center of mass of $S_1(u_h)$, lies on the $x_n$ axis, i.e. $$x^*=\frac {x^*_h \cdot e_n}{d_n} \,\, \,  e_n.$$

By construction, the restriction of $u_h$ to $S_1(u_h)$ satisfies
$$u_h \ge 0, \quad u_h(0)=0, \quad \nabla u_h(0)=0, \quad \det D^2 u_h=x_n^\alpha,$$
$$ u_h = \ph \quad \mbox{on $\p S_1(u_h)$,} \quad \mbox{with $\ph=1$ on $\p S_1 \setminus\{0\}$, and $\ph(0)=0$,}$$
and when we restrict to the $n-1$ dimensional space passing through $x^*$ we have
$$ B_1'  \subset S_1(u_h) \subset C(n) B'_1 \quad \quad \mbox{on the hyperplane $\{x_n=x^*\cdot e_n\}$.}$$

In order to prove that $u_h$ belongs to the class $\mathcal D_0^{\bar \mu} $ above it remains to show that $$c \le x^* \cdot e_n \le C.$$

We obtain this by choosing appropriate lower and upper barriers for $u_h$ in $S_1(u_h)$. Indeed, if $x^* \cdot e_n$ is very small then we obtain $u_h \ge w_3$ where $w_3$ is the barrier
$$w_3(x)=c|x'|^2 + c^{1-n} x_n^2 + t x_n,$$
for some $t>0$ and we contradict $\nabla u_h(0)=0$.
On the other hand if $x^* \cdot e_n$ is very large then $S_1(u_h)$ contains an ellipsoid $E$ (centered at $x^*$) of large volume and we contradict that $u_h \ge 0$ similarly as in \eqref{4}. This proves that $$u_h \in \mathcal D_0^{\bar \mu}(\infty,..,\infty) \quad \mbox{ for some $\bar \mu$ universal.}$$

The inequality $d_n \ge c h^{3/4}$ follows as in the proof of Proposition \ref{p2}, (see also proof of Lemma \ref{l4}). In order to prove the upper bound on $d_n$ we remark that $d_n \sim x^*_h \cdot e_n \sim b_u(h)$ where $b_u(h)$ represents the height of $S_h(u)$ i.e. $$b_u(h):=\max_{x \in S_h(u)} x_n.$$ By the compactness of the class $\mathcal D_0^\mu$ we easily obtain that $$\frac{b_v(1/2)}{b_v(1)} \le 1-c \quad \quad \mbox{for any $v\in \mathcal D_0^\mu$.}$$
This implies that $$\frac{b_u(h/2)}{b_u(h)} =\frac{b_{u_h}(1/2)}{b_{u_h}(1)} \le 1-c_1, \quad \quad \mbox{for all $h \le 1$},$$
with $c_1$ universal, hence $b_u(h) \le C h^\delta$ which finishes our proof.

\end{proof}

{\it Remark:}
The proof shows in fact that $\bar \mu$ depends only on $\alpha$ and $n$ since we can choose the constant $c$ in $w_3$ to depend only on $n$.
Also the inequality $b(u) \le C h^\delta$ implies that
\begin{equation}\label{31}
  u \ge c \, {x_n}^{1/ \delta}.
\end{equation}

Below we prove a similar lemma as above for functions $u \in \mathcal D_0^\mu(1,..,1,\infty,..,\infty)$. Before we state our lemma we introduce some notation.

\

{\bf Notation.} Fix $1\le k \le n-2$. We denote points in $\R^n$ by
$$x=(y,z,x_n) \quad \quad y:=(x_1,..x_k) \in \R^k, \quad z:=(x_{k+1},..,x_{n-1}) \in \R^{n-1-k}.$$

We say that a linear transformation $T:\R^n \to \R^n$ is a {\it sliding along $y$ direction} if
$$Tx=x + \nu_1z_1 +\ldots + \nu_{n-k-1}z_{n-k-1}$$ with
$$\nu_1,..,\nu_{n-k-1} \in span\{e_1,..,e_k\}.$$
We see that $T$ leaves the $(z,x_n)$ components invariant together with the subspace $(y,0,0)$. Clearly if $T$ is a sliding along the $y$ direction then so is $T^{-1}$ and $\det T=1$.

We will use the following linear algebra fact about transformations $T$ as above.

Assume $E_{x'} \subset \R^{n-1}$ is an ellipsoid in $x'=(y,z)$ variables (with center of mass at $0$). Then there exists $T$ a sliding along $y$-variable such that $$T E_{x'}=E_y \times E_z,$$ with $E_y$, $E_z$ two ellipsoids in the $y$ respectively $z$ variables. Here the ellipsoid $E_y$ is obtained by intersecting $E$ with the $y$-subspace. Using John's lemma we conclude that if $\Omega'\subset \R^{n-1}$ is a bounded convex set (with center of mass at the origin), there exists $T$ such that $T \Omega'$ is equivalent to a product of ellipsoids in the $y$ and $z$ variables i.e.
$$E_y \times E_z \subset T \Omega' \subset C(n) E_y \times E_z.$$

\begin{lem}\label{l5}
Assume Theorem \ref{T3} holds for all $l \le k-1$, for some $k$ with $1 \le k \le n-2$, and let
$$u \in \mathcal D_0^{\mu} (\underbrace{1,...,1}_{k \, \, times},\infty..,\infty).$$
For each $h \in (0,1]$, after a rotation (relabeling) of the $y$ respectively $z$ coordinates, there exist a sliding $T_h$ along the $y$ variable, and a sliding $A_h$ along $x_n=0$, and a diagonal matrix $D_h$
$$D_h=diag (d_1,d_2,..,d_n), \quad \quad \mbox{with} \quad \left(\prod_{i=1}^n  d_i^2 \right ) \, \, d_n^\alpha=h^n,$$
such that the normalized solution
$$u_h(x):=\frac 1 h u(T_h A_h D_h x) \quad \quad \mbox{satisfies} \quad u_h|_{S_1(u_h)} \in \mathcal D_0^{\bar \mu} (\underbrace{1,...,1}_{k \, \, times},\infty..,\infty).$$
Moreover if $x_h^*$ denotes the center of mass of $S_h(u)$ then,
$$c d_n \le x_h^* \cdot e_n \le C d_n, \quad  \quad \quad ch^{3/4} \le d_n \le C h^\delta,$$
 $$\mbox{and} \quad c \le d_i h^{-\frac 12} \le C \quad \mbox{for $i\le k$}.$$
The constants $\bar \mu$, $c$, $C$, $\delta$ above depend on $\mu$, $n$, $\alpha$ and $k$.
\end{lem}

The lemma states that if $u$ satisfies the hypothesis of Theorem \ref{T3} then we can normalize it in $S_h$ (using also a sliding along $y$ variable) such that the normalized solution satisfies essentially to same hypothesis as $u$.

\begin{proof}
 First we remark that if we use an affine deformation $x \to TAx$ with $T$ sliding along $y$, $A$ sliding along $x_n=0$, and let
 $$\tilde u(x):=u(TAx)$$ then the intersection of the $y$-subspace (passing through the center of mass) with $S_h(u)$ is left invariant. Precisely we have
$$\mathcal S_t'(\tilde u) \cap\{z=0 \} \quad = \quad \mathcal S_t'(u) \cap \{z=0\} \quad \quad \quad \mbox{for any $t>0$.}$$
For each $h$ we let $T=T_h$ and $A=A_h$ such that $S_h(\tilde u)$ has the center of mass $\tilde x^*_h$ on the $x_n$ axis and $$S_h(\tilde u) \cap \{x_n=\tilde x ^*_h \cdot e_n\}$$ is equivalent to a product of ellipsoids $E_y \times E_z$.
After a rotation of the $y$ respectively $z$ coordinates we may assume that $E_y$, $E_z$ have axes of lengths $d_1 \le \ldots \le d_k$ respectively $d_{k+1},..,d_{n-1}$ parallel to the coordinate axes. From the boundary data of $u$ we know that $$\{(0,0,y)| \quad |y| \le c h^{1/2}\} \subset S_h(u) $$ which implies $$d_i \ge c h^{1/2} \quad \mbox{for $i=1,..,k$.}$$
We choose $d_n$ as before in terms of $d_1,\ldots,d_{n-1}$ so that it satisfies the above identity for $D_h$. We let $$u_h=\frac 1h \tilde u (D_h x),$$
and obtain
$$u_h \ge 0, \quad u_h(0)=0, \quad \nabla u_h(0)=0, \quad \det D^2 u_h=x_n^\alpha,$$
$$ u_h = \ph \quad \mbox{on $\p S_1(u_h)$,} \quad \mbox{with $\ph=1$ on $\p S_1 \setminus\{G_y\}$, and $\ph=\psi(y)$ on $G_y$,}$$
where $$G_y:=\{(y,0,0)| \quad \psi(y) \le 1\}$$
and $\psi$ is a nonnegative function in $y$ satisfying
$$\mu^{-1} \mathcal N_y \ge D^2_y \psi \ge \mu \mathcal N_y \quad \quad \mathcal N_k=diag(a_1^2,..,a_k^2), \quad a_i:=d_i h^{-\frac 12} \ge c.$$
Moreover, by construction, when we restrict to the $n-1$ dimensional space passing through $x^*$ the center of mass of $S_1(u_h)$ we have
$$ B_1'  \subset S_1(u_h) \subset C(n) B'_1 \quad \quad \mbox{on the hyperplane $\{x_n=x^*\cdot e_n\}$.}$$

The properties above imply that $$u_h \in \mathcal D_0^{\tilde \mu}(a_1,..,a_k, \infty,..,\infty) \quad \quad \mbox{for some $\tilde \mu$ universal.}  $$
Indeed, for this it suffices to prove that $$c \le x^* \cdot e_n \le C,$$ and this follows exactly as in the proof of Lemma \ref{l6}.

Since $u_h$ belongs to the class above, the bounds on $d_n$ follow in the same way as in Lemma \ref{l6}.

It remains to show that $a_i$, $1\le i \le k$ remain bounded above by a universal constant for all $h$.
From our hypothesis and Remark \ref{r2} we know that Proposition \ref{p3} holds for all $l$ with $l \le k-1$. Using compactness as in Lemma \ref{l4.1} this implies that the property $P(l)$ holds for all $l \le k-1$. Precisely,

there exists $C_0:=C_0(M,\mu,k,n,\alpha)$ such that if
$$v \in D_0^\mu(a_1,..,a_{n-1}), \quad \quad \mbox{with $a_l \ge C_0$, for some $l \le k$},$$
and some $a_i \in [\mu,\infty) \cup \{\infty \}$, then
$$\mathcal S_t'(v) \subset \{|(x_l,..,x_{n-1})| \le \frac 1 M \},$$
for some $t\in [c_k,1]$ with $c_k$ depending on the parameters above.

Now we argue as in Lemma \ref{l3.1}. For $i \le k$ denote by $d_i(h)$ and $a_i(h)$ the quantities $d_i$ and $a_i=d_i h^{-1/2}$ constructed above that correspond to the section $S_h(u)$. Notice that $a_i(h)$ represent the lengths of the axes of a $k$-dimensional ellipsoid (in the $y$ variable) which is equivalent to $\mathcal S_h'(u) \cap \{z=0 \}$. We show that for any $h$ we have
$$
\max a_i(h) \ge \bar C \quad \quad \Rightarrow \quad \quad \max a_i(th) \le \frac 12 \max a_i(h),$$
for some $t \in [\bar c, 1]$, and with $\bar C$ universal. Since $\max a_i(1)$ is bounded above by a universal constant we easily obtain that $\max a_i(h) $ remains bounded above.

Let $C_0$ and $c_k$ denote the constants in the property above for $\tilde \mu$ and $M=4 \sqrt n$, hence $C_0$, $c_k$ are universal. Assume that $a_l(h)$ is the first $a_i(h)$ greater than $C_0$ i.e.
$$a_1 \le \cdots \le a_{l-1} \le C_0 \le a_l  \le \cdots \le a_{k}.$$
Since $u_h \in \mathcal D_0^{\tilde \mu} (a_1,..,a_k, \infty,..,\infty)$, we have (see \eqref{10})
$$ \mathcal S_t'(u_h) \subset \{|(x_1,..,x_{l-1})| \le C_1\} \times \{|(x_l,..,x_{n-1})| \le \frac 1 M\},$$
for some $C_1(c_k)$ universal.
Since $$\{ y| (y,0) \in \mathcal S'_{th}(u) \}\,  = \,diag(a_1,..,a_k)  \, \, \, \{y|  (y,0) \in \mathcal S_t'(u_h)\},$$ we obtain
$$\mathcal S'_{th}(u) \cap \{z=0 \} \subset \prod_{i=1}^{l-1}\{|x_i|\le C_1 a_i\} \times \prod_{i=l}^{k}\{|x_i|\le \frac{a_i}{M}\} \times \{z=0\}.$$
For $i \le l-1$ we have
$$C_1 a_i \le C_1 C_0 := \frac {\bar C} { M} \le \frac{\max a_i} {M},$$
 and we find
$$\mathcal S'_{th}(u) \cap \{z=0\} \subset \{ |y| \le \frac 14 \max a_i(h) \} \times \{z=0\},$$
which gives $$\max a_i(th) \le \frac 12  \max a_i(h).$$
\end{proof}

\section{Pogorelov type estimates}\label{s5}

In this section we obtain two estimates of Pogorelov type that will be used in Section \ref{s6} for the proof of Theorem \ref{T3}. They appeared also in \cite{S2} where the obstacle problem for Monge-Ampere equation was investigated. 

\begin{thm}\label{Po}
Assume $u\in C^4(\Omega) \cap C(\ov \Om)$ is convex, $u=0$ on $\p \Om$,
$$\det D^2 u=f(x_2,\ldots,x_n) \quad \mbox{in $\Omega$}, \quad \quad f>0.$$
Then
$$u_{11}|u| \le C(n,\max_{\Omega}|u_1|).$$
\end{thm}

{\it Remark:} The constant $C(n,\max_{\Omega}|u_1|)$ does not depend on $f$ or $\Omega$.

\begin{proof} We may assume that $u \in C^4(\ov \Om)$ since we apply the estimate to $u+\eps$ and then let $\eps \to 0$. We write
$$ \log \det D^2 u= \log f$$
and differentiate with respect to $x_1$
\begin{equation}{\label{48}}
u^{ij}u_{1ij}=0,
\end{equation}
where $[u^{ij}]=[D^2u]^{-1}$ and we use the index summation convention.
Differentiating once more, we have
\begin{equation}{\label{49}}
u^{ij}u_{11ij}-u^{ik}u^{jl}u_{1ij}u_{1kl}=0.
\end{equation}

Suppose the maximum of
\begin{equation}{\label{410}}
\log u_{11} + \log|u| + \frac{1}{2}|u_1|^2=M
\end{equation}
occurs at the origin. One can also assume that $D^2u(0)$ is diagonal since
the the domain transformation (sliding along $x_1$ variable)
\begin{equation}{\label{412.5}}
\tilde{u}(x_1,..,x_n):=u(x_1-\alpha_2 x_2-..-\alpha_n x_n, x_2, ..,x_n), \quad
\alpha_i=\frac{u_{1i}(0)}{u_{11}(0)}
\end{equation}
does not affect the equation or the maximum
in (\ref{410}). Thus, at $0$

\begin{equation}{\label{411}}
\frac{u_{11i}}{u_{11}}+\frac{u_i}{u}+u_1u_{1i}=0
\end{equation}

\begin{equation}{\label{412}}
\frac{u_{11ii}}{u_{11}}-\frac{u_{11i}^2}{u_{11}^2}+\frac{u_{ii}}{u}
-\frac{u_i^2}{u^2}+u_{1i}^2+u_1u_{1ii}\le 0
\end{equation}

We multiply (\ref{412}) by $u_{ii}^{-1}$ and add

$$\frac{u_{11ii}}{u_{11}u_{ii}}-\frac{u_{11i}^2}{u_{ii}u_{11}^2}+\frac{n}{u}
-\frac{u_i^2}{u_{ii}u^2}+\frac{u_1u_{1ii}}{u_{ii}}+u_{11}\le 0.$$

From (\ref{411}) we obtain

$$ \frac{u_i}{u}=-\frac{u_{11i}}{u_{11}}, \quad i \ne 1,$$
which together with (\ref{48}), (\ref{49}) gives

$$\sum_{i,j\ne
1}\frac{u_{1ij}^2}{u_{11}u_{ii}u_{jj}}+\frac{n}{u}
-\frac{u_1^2}{u_{11}u^2}+u_{11}\le 0,$$
thus,
$$e^{2M}-ne^{\frac{1}{2}u_1^2}e^M \le u_1^2e^{u_1^2}$$
and the result follows.

\end{proof}

The second estimate deals with curvature bounds for the level sets of solutions to certain Monge-Ampere equations.

We assume the convex function $u\in C^4(\Omega) \cap C(\ov \Om)$ is increasing in the $e_n$ direction and
\begin{equation}\label{4.13}
u=\sigma x_n \quad \mbox{on $\p \Om$,}
\end{equation}
for some $\sigma>0$. We denote by
$v(x_1,..,x_{n-1},s)$
the graph in the $-e_n$ direction of the $s$ level set, i.e
$$u(x_1,..,x_{n-1},-v(x_1,..,x_{n-1},s))=s.$$
Clearly $v$ is convex.

\begin{thm}{\label{abo}}
Assume $u$ satisfies \eqref{4.13} and
$$ u_n^\alpha \, \, \det D^2 u =f(x_2,..,x_{n-1},u)  \quad \mbox{in $\Om$,}$$
for some $\alpha \ge 0$.
Then
$$v_{11} \, \, |u-\sigma x_n| \le C \left ( n,\alpha,\sigma,\max_\Om u_n, \max_\Om |v_1| \right).$$

\end{thm}

{\it Remark:} The constant $C$ does not depend on $f$ or $\Om$. We also have the equality $$|u-\sigma x_n|= |\sigma v+ s|.$$

First we write the equation for $v$.
The normal map to the graph of $u$ at
$$X=(x_1,..,x_n,x_{n+1})=(x_1,..,x_n,u(x))$$ is given by
$$\nu=(\nu_1,..,\nu_{n+1})=(1+|\nabla
u|^2)^{-\frac{1}{2}}(-u_1,..,-u_n,1).$$
The Gauss curvature of the graph of $u$ at $X$ equals
\begin{align*}
K(X) &=\det D_i \left ( u_j (1+|\nabla u|^2)^{-\frac{1}{2}} \right)\\
&=(1+|\nabla u|^2)^{-\frac{n+2}{2}} \det D^2 u \\
&= (\nu_{n+1})^{n+2} \det D^2 u.
\end{align*}

The graph of $u$ can be viewed as the graph of $v$ in the $-e_n$ direction,
thus
$$K(X)= (-\nu_n)^{n+2} \det D^2 v,$$
which gives
$$ \det D^2 u= \det D^2 v \left( \frac{-\nu_n}{\nu_{n+1}}  \right)^{n+2}.$$
Since
$$u_n=\frac{-\nu_n}{\nu_{n+1}}=-\frac 1 {v_s}=\frac{1}{|v_s|},$$
we find
$$u_n^\alpha \, \det D^2 u= |v_s|^{-(n+2+\alpha)} \, \det D^2 v.$$

By abuse of notation we relabel the $s=x_{n+1}$ variable (i.e. the last coordinate of $v$) by $x_n$ we find that $v$ satisfies
$$\det D^2 v=f(x_2,x_3,..,x_{n-1},x_n) |v_n|^{n+2+\alpha}, \quad \quad \quad v_n<0,$$
and it is defined in
$$\Omega_v:=\left \{v < -\frac {x_n} {\sigma} \right \}, \quad \quad  v=-\frac {x_n} {\sigma} \quad \mbox{on $\p \Om_v$.}$$
We denote by
$$w:=v + \frac {x_n} {\sigma},$$
thus $w=0$ on $\p \Omega_v$ and in $\Om_v$
$$ \det D^2 w= f(x_2,..,x_n) \left (\frac 1 \sigma - w_n \right )^{n+2+\alpha},$$
and also
$$\left (\frac 1 \sigma -w_n \right )^{-1}=u_n>0, \quad \quad w_1=v_1, \quad w_{11}=v_{11}.$$
In order to prove Theorem \ref{abo} it suffices to prove the next
estimate.

\begin{lem}{\label{above}}
Suppose that in the bounded set $\{w<0 \}$
$$ \det D^2 w= f(x_2,..,x_n)(w_\xi + \beta)^{n+2+\alpha}, \quad w_\xi + \beta>0,$$
where $\xi$ is some vector. Then
$$w_{11}|w| \le C \left (n,\alpha,\max_{\{w<0 \}} |w_1|, \max_{\{w<0 \}}
\frac{\beta}{w_\xi +
\beta} \right ).$$
\end{lem}

\begin{proof}
Assume the maximum of
\begin{equation}{\label{415}}
\log w_{11} + \log |w| + \frac{\eta}{2}w_1^2
\end{equation}
occurs at the origin, where $\eta>0$ is a small constant depending only on
$\max|w_1|$, to be made precise later.

Again we can assume that $D^2w(0)$ is diagonal. Indeed, using
a sliding in the $x_1$ direction as in (\ref{412.5}) we find that the transformed function $\tilde w$ satisfies
$$\det D^2 \tilde{w}=f(x_2,..,x_n) (\tilde{w}_{\tilde{\xi}} + \beta)^{n+2+\alpha}, \quad
\tilde{w}_{\tilde{\xi}}(\tilde x)=w_\xi(x),$$
thus, the hypothesis and the conclusion remain invariant under this
transformation.

We write
$$\log \det D^2 w= \log f (x_2,..,x_n)+ \gamma \log (w_\xi +\beta),$$
with $$\gamma:=n+2+\alpha.$$
Taking derivatives in the $e_1$ direction we find

\begin{equation}{\label{416}}
\frac{w_{1ii}}{w_{ii}}=\gamma \, \frac{w_{1\xi}}{w_\xi+\beta}
\end{equation}

\begin{equation}{\label{417}}
\frac{w_{11ii}}{w_{ii}}-\frac{w_{ij1}^2}{w_{ii}w_{jj}}=
\gamma \, \frac{w_{11\xi}}{w_\xi+\beta}- \gamma\, \frac{w_{1\xi}^2}{(w_\xi+\beta)^2}
\end{equation}

On the other hand, from (\ref{415}) we obtain at $0$

\begin{equation}{\label{418}}
\frac{w_{11i}}{w_{11}}+\frac{w_i}{w}+\eta w_1w_{1i}=0,
\end{equation}

\begin{equation}{\label{419}}
\frac{w_{11ii}}{w_{11}}-\frac{w_{11i}^2}{w_{11}^2}+\frac{w_{ii}}{w}
-\frac{w_i^2}{w^2}+\eta w_1w_{1ii} +\eta w_{1i}^2 \le 0.
\end{equation}

We multiply (\ref{419}) by $w_{ii}^{-1}$ and add, then use (\ref{417}),
(\ref{416})

\begin{equation}{\label{420}}
\frac{1}{w_{11}}\left( \frac{w_{ij1}^2}{w_{ii}w_{jj}}
+ \gamma \frac{w_{11\xi}}{w_\xi+\beta}-\gamma \frac{w_{1\xi}^2}{(w_\xi+\beta)^2}
\right)-
\end{equation}

$$-\frac{w_{11i}^2}{w_{ii}w_{11}^2}+\frac{n}{w}
-\frac{w_i^2}{w_{ii}w^2}+\gamma \eta w_1\frac{w_{1\xi}}{w_\xi+\beta} +\eta
w_{11} \le 0.$$

Since $\gamma \ge n$ we have

\begin{equation}\label{421}
\sum_{i,j \ne
1}\frac{w_{ij1}^2}{w_{ii}w_{jj}}- \gamma \frac{w_{1\xi}^2}{(w_\xi+\beta)^2} \ge \\
\end{equation}

$$-\frac{w_{111}^2}{w_{11}^2}+\frac{1}{n}\left( \sum_1^n
\frac{w_{1ii}}{w_{ii}}
\right)^2-\gamma \frac{w_{1\xi}^2}{(w_\xi+\beta)^2}\ge$$

$$ -\frac{w_{111}^2}{w_{11}^2}+
\frac{\gamma^2}{n}\frac{w_{1\xi}^2}{(w_\xi+\beta)^2}
-\gamma \frac{w_{1\xi}^2}{(w_\xi+\beta)^2} \ge$$

$$\ge -\frac{w_{111}^2}{w_{11}^2}.$$

From (\ref{418})
$$\frac{w_i}{w}=-\frac{w_{11i}}{w_{11}}, \quad \mbox{for $i\ne 1$}$$
which together with (\ref{421}) gives us in (\ref{420})

\begin{equation}{\label{422}}
\frac{\gamma}{w_\xi+\beta}\left(\frac{w_{11\xi}}{w_{11}}+
\eta w_1w_{1\xi} \right) - \frac{w_{111}^2}{w_{11}^3}+
 \frac{n}{w}-\frac{w_1^2}{w_{11}w^2}+\eta w_{11} \le 0.\end{equation}

From (\ref{418})

\begin{equation}{\label{423}}
\frac{w_{11\xi}}{w_{11}}+ \eta w_1w_{1\xi}=-\frac{w_\xi}{w}
\end{equation}
and also
$$\frac{w_{111}}{w_{11}}=-\frac{w_1}{w} - \eta
w_1w_{11}$$
thus,
\begin{equation}{\label{424}}
\frac{w_{111}^2}{w_{11}^2} \le 2 \frac{w_1^2}{w^2} +  2 \eta^2
w_1^2w_{11}^2.
\end{equation}

We use \eqref{423}, \eqref{424} in \eqref{422} and obtain

$$-\frac{\gamma w_\xi}{(w_\xi +\beta)w}+ \frac{n}{w} -
3\frac{w_1^2}{w_{11}w^2}+\eta(1- 2\eta w_1^2) w_{11} \le 0.$$

Multiplying by $w_{11}w^2$ we have

$$\eta(1- 2 \eta w_1^2) (ww_{11})^2+\left( n-\gamma + \gamma \frac{\beta}{w_\xi +\beta}
\right) ww_{11} \le 3 w_1^2$$
and the result follows if $\eta$ is chosen such that $\eta \, (\max w_1^2) <1/4$.

\end{proof}

\section{Proof of Theorem \ref{T3}}\label{s6}

We prove Theorem \ref{T3} by induction on $k$. The cases $k=0$ and the induction step $k-1\Rightarrow k$ are quite similar.
We start with $k=0$.
\begin{prop}\label{p4}
Theorem \ref{T3} holds for $k=0$. Precisely if $$u \in \mathcal D_0^\mu(\infty,..,\infty),$$
then
$$\mathcal S_h'(u) \subset \{|x'| \le C h^\beta  \}, \quad \quad \beta:=\frac{1}{2(n+1+\alpha)},$$
for some $C$ depending on $\mu$, $n$ and $\alpha$.
\end{prop}

 \begin{rem} \label{r4}In view of Lemma \ref{l6} we may assume, after relabeling $\mu$, that all renormalized solutions $u_h$ given in Lemma \ref{l6} are in the same class $\mathcal D_0^\mu(\infty,..,\infty).$
\end{rem}
We prove Proposition \ref{p4} by studying the behavior of the tangent cone of $u$ at the origin.

The {\it tangent cone} $\Gamma_u$ of $u$ at the origin is obtained by taking the supremum of all supporting planes of $u$ at the origin. In other words the upper graph of $\Gamma_u$ is obtained by the intersection of all half-spaces that pass through the origin and contain the upper graph of $u$, therefore $\Gamma_u$ is lower semicontinuous.

We define the $n-1$ dimensional function $\gamma_u(x')$ as being the restriction of $\Gamma_u$ to $x_n=1$ i.e
$$\gamma_u(x'):= \Gamma_u(x',1).$$
By construction the upper graph of $\gamma_u$ is a closed set and
$$u(x) \ge \Gamma_u(x)=x_n \gamma_u(\frac{x'}{x_n}).$$
Since $\nabla u(0)=0$ we have $\gamma_0 \ge 0$ and $\inf \gamma_u=0$. In the next lemma we obtain some useful properties of $\gamma_u$.

\begin{lem}\label{l7}

a) $$\gamma_u(x') \ge c_0|x'|-C_0.$$

b) If $$\gamma_u \ge C_0 p' \cdot (x'-x_0'),$$for some unit vector $p'\in \R^{n-1}$, $|p'|=1$ and some $x_0'$ then
$$\gamma_u \ge p' \cdot (x'-x_0') + c_0.$$
The constants $c_0$, $C_0$ above are universal constants.
\end{lem}

\begin{proof}
a) We compare $u$ with
$$w:=c x' \cdot p' + c|x'|^2 + C (x_n^2 -\mu^{-1} x_n),$$
with $p'$ a unit vector.
We choose $c$ small such that $w \le 1$ in $\Om \subset B^+_{1/\mu}$ and $C$ large such that $\det D^2 w \ge \det D^2u$. We find $u \ge w$ hence
$$\gamma_u \ge \gamma_w=c_0 x' \cdot p'-C_0 x_n,$$
which proves part a).

b) Assume that $p'=e_1$ and let $x_0' \cdot e_1=q$. Then
$$u \ge x_n\gamma_u(\frac{x'}{x_n}) \ge C_0(x_1-q x_n)^+.$$
In the set $O=\Omega \cap \{x_1-q x_n >-1\}$ we compare $u$ with
$$w:= \frac {C_0}{2} (x_1-qx_n) + \frac{C_0}{8} (x_1-q x_n)^2 + \delta (x_2^2+\ldots + x_n^2) + \delta x_n,$$
where $\delta$ is small, fixed, depending on $\mu$.
Notice that if $C_0$ is sufficiently large we have $\det D^2 w \ge \det D^2 u$ and $w \le u$ on $\p O$.
Indeed, on $\p O \setminus \p \Om$ we have $$x_1-q x_n=-1 \quad \Rightarrow \quad w \le 0 \le u,$$ and in the set $\p O \cap \p \Omega$,
$$1 \ge u \ge C_0(x_1 - q x_n) \quad \Rightarrow \quad w \le 1.$$
From $w(0)=0$ and the inequalities above we obtain $w \le u$ on $\p \Om$.
 In conclusion $$\gamma_u \ge \gamma_w \ge \frac{C_0}{2}(x_1-q) + \delta.$$

\end{proof}

\begin{rem}\label{r5}
From the proof we see that we only need the weaker assumption $$u \ge C_0(x_1-q x_n) \quad \mbox{ on $\p \Omega$,}$$
in order to obtain the conclusion of part b).
\end{rem}

From part a) we see that $x_o' \in \R^{n-1}$ the point where $\gamma_u$ achieves its infimum belongs to $B_C'$. As a consequence of Lemma \ref{l7} we obtain the following corollary about the section $S_1(\gamma_u) \subset \R^{n-1}$.

\begin{cor}\label{c1} There exist universal constants $c_*$ small, $C_*$ large, such that
$$B_{c_*}' \subset S_1(\gamma_u) - x_o' \subset B_{C_*}'$$
$$S_{c_*}(\gamma_u)-x_o' \subset (1-c_*) (S_1(\gamma_u)-x_o').$$
\end{cor}

\begin{proof}
We only need to show that $\gamma_u$ cannot be too small near $\p S_1(\gamma_u)$.
Assume by contradiction that $\gamma_u(y_0') \ll 1$ for some $y_0'$ near $\p S_1(\gamma_u)$. Then we can find a plane of slope $C_0$ i.e. $C_0 p'\cdot (x'-x_0') \le \gamma_u(x')$ with $x_0' \in \p S_1(\gamma_1)$ sufficiently close to $y_0'$. We apply part b) of Lemma \ref{l7} and obtain that $\gamma_u$ is greater than a universal constant in a neighborhood of $x_0'$ and we reach a contradiction.

\end{proof}

Next we apply the corollary above for the rescalings $u_h$ of $u$ defined in Lemma \ref{l6} (see Remark \ref{r4}). For any $h \in (0,1]$, we have
$$u_h(x)=\frac 1 h \tilde u(D_h x), \quad \quad \mbox{with} \quad \tilde u(x)=u(A_hx).$$
Notice that $\gamma_u$ is just a translation of $\gamma_{\tilde u}$.
Since
$$\Gamma_{u_h}(x)=\frac 1 h \Gamma_{\tilde u}(D_h x)$$
we divide by $x_n$ and obtain
$$ \gamma_{u_h} \left (\frac{x'}{x_n}\right )=\frac {d_n}{ h } \, \, \gamma_{\tilde u}\left(\frac{ D_h' x'}{d_n x_n}\right )$$
or
$$\gamma_{u_h}(x')=\frac{d_n}{h} \, \, \gamma_{\tilde u}\, (d_n^{-1} \, D_h'x') \quad \quad \mbox{with} \quad D_h'=diag(d_1,..,d_{n-1}).$$

This implies that $$d_n^{-1}D_h'\, \,  S_s(\gamma_{u_h})= S_{s \, h/d_n}(\gamma_{\tilde u}).$$

We apply Corollary \ref{c1} for the sections $S_1$ and $S_{c_0}$ of $\gamma_{u_h}$ and use also that $\gamma_u$ is a translation for $\gamma_{\tilde u}$. We obtain the following inclusions for the sections $S_t(\gamma_u)$, $t:=h/d_n$
\begin{equation}\label{51}
d_n^{-1} D_h' \, B_{c_*}' \quad \subset \quad S_t(\gamma_u) -x_o' \quad  \subset \quad d_n^{-1} D_h' \, B_{C_*}', \quad  \quad \quad t=\frac{h}{d_n},
\end{equation}
and
\begin{equation}\label{52}
S_{c_* t}(\gamma_u)-x_o' \quad \subset \quad (1-c_*)\left(S_t(\gamma_u) -x_o'   \right).
\end{equation}

From Lemma \ref{l6} we know that as $h$ ranges from 1 to 0 the parameter $t=h/d_n$ covers an interval $[0,c]$. The inclusion \eqref{51} says that the sections $S_t(\gamma_u)$ are {\it balanced} around the minimum point $x_o$, i.e. there exists an ellipsoid $E$ such that
$$c_0 E \subset S_t(\gamma_u)-x_o' \subset C_0 E, \quad \quad \quad E=d_n^{-1}D_h B_1'.$$
A dilation of the ellipsoid $E$ above is equivalent also to the normalized diameter $\mathcal S'(u)$. Indeed, from the definition of $D_h$
\begin{equation}\label{53}
h^{-1/2} D_h' \, B_1' \, \, \subset \, \, \mathcal S_h'(u) \, \, \subset \, \, C(n) h^{-1/2} D_h' \, B_1'.
 \end{equation}
 The inclusions \eqref{51}, \eqref{53} show the relation between the sections $S_t(\gamma_u)$ and $\mathcal S_h'(u)$.

 Property \eqref{52} implies that
 \begin{equation}\label{54}
\gamma_u(x') \ge c |x'-x_o'|^M \quad \quad \mbox{in $B_c'(x_o)$,}
\end{equation}
for some $M$ large universal. From the fact that the sections of $\gamma_u$ are balanced one can also prove that $\gamma_u \in C^{1,\beta}$. Thus each small section $S_t(\gamma_u)$ contains a small ball of radius $t^{1/(1+\beta)}$ and it is contained in a ball of radius $t^{1/M}$. However, these bounds are not sufficient for the proof of Proposition \ref{p4}.

We remark that so far in the proof we only used that the Monge-Ampere measure of $u$ is bounded by above and below by multiples of $x_n^\alpha$. Below we use the estimates of Section 4 and the fact that the Monge-Ampere measure is precisely $x_n^\alpha$, and conclude that $\gamma_u$ has quadratic growth near $x_o'$.
Precisely we show the following.

\begin{lem}\label{l8}
There exists universal constants $c_1$, $C_1$ such that
$$c_1|x'-x_o|^2 \le \gamma_u(x') \le C_1 |x'-x_o'|^2 \quad \quad \mbox{in $B_{c_1}'(x_o')$}.$$
\end{lem}

This estimate for $\gamma_u$ easily implies Proposition \ref{p4}. Indeed, the lemma gives $$t^{1/2}B_c' \subset S_t(\gamma_u) -x_o' \subset t^{1/2} B_C',$$
which together with \eqref{51} implies that for all $i<n$,
$$ c d^{1/2}_n \le d_i h^{-1/2} \le C d^{1/2}_n.$$
Then
$$\left(\prod_{i=1}^{n-1} d_i^2 \right) \, d_n^{2+\alpha}=h^n \quad \Rightarrow \quad ch \le {d_n}^{n+1+\alpha} \le Ch  $$
and Proposition \ref{p4} follows from \eqref{53}.

\

Below we prove Lemma \ref{l8}. After performing a sliding along $x_n=0$ of bounded norm, we may assume that $x_o'=0$.

\

{\it Step 1:} In step 1 we use Theorem \ref{Po} in the set $\{ u < c x_n\}$ to obtain
\begin{equation}\label{55}
D^2\gamma_u \le C I \quad \quad \mbox{in} \quad \{\gamma_u <c \}.
\end{equation}

In order to apply Theorem \ref{Po} for $u$ we first need to bound $|\nabla u|$ in the set $\{u<c x_n\}$ for some $c$ small. To this aim we observe that the projection of $\Om=S_1(u)$ along $e_n$ into $\R^{n-1}$ contains the ball $B_{1/C_0}'$, with $C_0$ as in Lemma \ref{l7}. Otherwise we can find a direction, say $x_1$ such that
$$\Om \subset \{ x_1 \le 1/C_0 \},$$
hence $$u \ge C_0 x_1 \quad \mbox{on $\p \Om$},$$
and by Lemma \ref{l7} (see Remark \ref{r5}) we find $$\gamma_u \ge x_1 + c_0,$$ and we contradict that $\gamma_u(0)=0$ (since $x_o'=0$.)

Since $ \Omega \subset B_C^+$ contains a ball $B_\mu(x^*)$, and its projection contains $B_{1/C_0}'$ and $0 \in \p \Om$, we see from its convexity that it must contain also $B_r(\varrho e_n)$ for some small fixed universal constants $\varrho$, $r$. Now we use $u \ge 0$, $u(0)=0$ and conclude that $|\nabla u| \le C$ in the convex set generated by $0$ and $B_{r/2}(\varrho e_n)$. On the other hand by \eqref{31} and \eqref{54} we see that this convex set contains the set $\{ u < cx_n \}$ if $c$ is sufficiently small.

 Now let $w:=u-cx_n$ and notice that the rescalings
 $$w_\lambda(x):=\frac{1}{\lambda} w(\lambda x), \quad \quad \det D^2 w_\lambda = c(\lambda) x_n^\alpha,$$
 have the same gradient bound in $\{w_\lambda <0\}$ and they converge uniformly on $x_n=1$ to $|\gamma_u - c|$. By Theorem \ref{Po} we find
 $$|w_\lambda| \, \, \p_{11}w_\lambda \, \le C,$$
 hence
 $$|\gamma_u-c| \, \, \p_{11} \gamma_u  \, \le C,$$
 which proves step 1.

\

In the course of the proof we showed also that the segment $[0,\varrho x_o] \subset S_1(u)$ with $x_o:=(x_o',1)=e_n$. We apply this for the rescaling $u_h$ and obtain
$$[0,\varrho d_n  \, x_o] \subset S_h(u).$$ Using the bounds on $d_n$ from Lemma \ref{l7} we find (see also \eqref{31})
\begin{equation}\label{57}
 c t^{1/\delta} \le  u(t x_o',t) \le C t^{4/3}.
 \end{equation}
We can extend this inequality at points $y'$ near $x_o'$,
\begin{equation}\label{56}
c t^{1/\delta} \le u(t y',t) - t \gamma_u(y') \le C t^{4/3} \quad \quad \mbox{for all $y' \in B_c'$.}
\end{equation}
Indeed, \eqref{56} follows by applying \eqref{57} to the function $u-p \cdot x$ where $p \cdot x$ is the linear function which restricted to $x_n=1$ becomes tangent by below to $\gamma_u$ at $y'$. From Step 1 we see that when $|y'|$ is small, the slope of $l$ is also small and $u-l$ (renormalized at its $1/2$ section) belongs to a class $\mathcal D_0^c (\infty,...,\infty)$. Therefore we can apply \eqref{57} for $u-l$ and obtain the desired inequality \eqref{56}.

\

{\it Step 2:} In step 2 we apply Theorem \ref{abo} for the Legendre transform of $u$ and obtain
$$D^2 \gamma_u \ge c I \quad \quad \mbox{in $B_c'$.}$$

 Let $u^*$ denote the Legendre transform of $u$,
 $$u^*(\xi):= \sup_{x \in \ov \Om} \left( x \cdot \xi - u(x)\right).$$
Since $u$ is lower semicontinuous the supremum is always achieved at some point $x \in \ov \Om$. We are interested in the behavior of $u^*(\xi)$ for $|\xi| \le c$ small. From the boundary values of $u$ we see that the maximum is realized either at $0$ or at some $x \in \Om$, and clearly $u^* \ge 0$. We define $K$ as the convex set
$$K:=\{ u^*=0\}.$$
 If $\xi \in K$ then the maximum is achieved at $0$, and this happens if and only if $$\xi' \cdot x' + \xi_n \le \gamma_u(x') \quad \mbox{for all $x'$} \quad \quad \Leftrightarrow \quad \xi_n \le - \gamma_u^* (\xi'),$$
where $\gamma_u^*$ represents the Legendre transform of $\gamma_u$.
In conclusion
$$K=\{\xi_n \le - \gamma_u^* (\xi')  \}.$$
From Step 1 and \eqref{54} we know that
$$c|x'|^M \le \gamma_u(x') \le C |x'|^2  \quad \mbox{in $B_c'$}$$
hence
\begin{equation}\label{58}
\{ \xi_n \le -C |\xi'|^\frac{M}{M-1} \}  \subset K \subset \{ \xi_n  \le - c |\xi'|^2 \} \quad \mbox{in $B_c$}.
\end{equation}

 Since $u$ is strictly convex in $\Om$ we obtain that
 \begin{equation}\label{59}
 u^* \in C^1(B_c),
 \end{equation}
 and in the set $\{u^*>0\}$ we have
 $$\det D^2 u^*(\xi) =(\det D^2 u (x))^{-1}=x_n^{-\alpha}=(u^*_n(\xi))^{-\alpha},$$
thus, $u^*$ solves the equation
\begin{equation}\label{59.5}
(u^*_n)^\alpha \det D^2 u^* =1 \quad \quad \mbox{in $B_c \setminus K$.}
\end{equation}
Also from \eqref{56} and the definition of Legendre transform we find
$$ u^*(\xi) \ge c \left ( (\xi_n + \gamma_u^*(\xi'))^+ \right )^4 \quad \quad \mbox{and} \quad |\nabla u^*| \le 1 \quad \mbox{in $B_c$},$$
which together with \eqref{58} implies that
$$O:= \{ u^* < \eta (\xi_n + \eta) \} \subset B_{c_1},$$
with $\eta$ and $c_1$ sufficiently small universal constants. Moreover, in $O$, the Lipschitz norms of the level sets of $u^*$ (viewed as graphs in the $-e_n$ direction) are bounded by a universal constant.

 Next we apply Theorem \ref{abo} for $u^*$ in $O$ and obtain universal bounds for the second derivatives of the level sets of $u^*$ in a fixed neighborhood of the origin. Writing this for $K$, the $0$ level set, we obtain the desired result of Step 2 since, in a neighborhood of $0 \in \R^{n-1}$, $$D^2 \gamma_u^* \le C I \quad \Rightarrow \quad D^2 \gamma_u \ge c I.$$

We cannot apply directly Theorem \ref{abo} since $u^*$ is not strictly increasing in the $e_n$ direction. However we show below using approximations that the theorem still applies in our case i.e. for functions $u^*$ that satisfy \eqref{59}, \eqref{59.5}. Formally, we write that $u^*$ solves the equation in $B_c$ with right hand side $f(u^*)$ with $f=\chi_{(0,\infty)}$, and then apply Theorem \ref{abo}.

\

{\it Approximation.} Define $$v_\eps= \max \left \{ u^*, \, \, \eps \left (1+ \xi_n + \frac 1 2 |\xi|^2\right ) \right \}$$
and we remark that in $B_{c_1}$, $v_\eps>0$, it is strictly increasing in the $e_n$ direction, $|\nabla v_\eps| \le 1$ and its level sets have Lipschitz norm bounded by a universal constant. In the set $\{v_\eps>u^*\} \cap B_{c_1}'$ we have
$$(\p_n v_\eps)^ \alpha \det D^2 v_\eps = \eps^{n+\alpha},$$
hence
$$ (\p_n v_\eps)^ \alpha \det D^2 v_\eps \ge f_\eps(v_\eps)  \quad \mbox{in $B_{c_1}$,} $$
in viscosity sense, with
$f_\eps$ a nondecreasing function satisfying $$f(s)=\eps^{n+\alpha} \quad \mbox{if $s \le \eps^{1/2}$}, \quad \quad f(s)=1 \quad \mbox{if $s \ge 2 \eps^{1/2}$.} $$

 We define $\bar v_\eps$ as the viscosity solution to
 $$(\p_n \bar v_\eps)^\alpha \det D^2 \bar v_\eps =f_\eps(\bar v_\eps) \quad \quad \mbox{in}\quad O_\eps:= \{v_\eps< \eta(\xi_n + \eta)\},$$
 $$ \bar v_\eps=v_\eps \quad \mbox{on $\p O_\eps$.} $$

 The existence of $\bar v_\eps$ follows by Perron's method and since $v_\eps$ is a subsolution, we have $\bar v_\eps \ge v_\eps$. This implies that $\bar v_\eps$ is strictly increasing in the $e_n$ direction and, $|\nabla \bar v_\eps|$ and the Lipschitz norm of the level sets of $\bar v_\eps$ are bounded by a universal constant. Therefore we can apply Theorem \ref{abo} for $\bar v_\eps$ in $O_\eps$ and obtain the uniform second derivative bounds for its level sets around the origin.

 It remains to show that $\bar v_\eps$ converges to $u^*$. Assume that a subsequence of $\bar v_\eps$ converge to $\bar v_0$. Then $\bar v_0$ is defined in $O$, $\bar v_0= u^*$ on $\p O$, and by construction $$\bar v_0 \ge u_*.$$

 We prove that also $\bar v_0 \le u^*$. Assume by contradiction that the maximum of $\bar v_0 -u^*$ is positive and occurs at a point $\xi_0$. From the convergence of $\bar v_\eps$ to $\bar v_0$ we obtain
 $$ (\p_n \bar v_0)^\alpha \det D^2 \bar v_0 =1 \quad \quad \mbox{in the set $\{ \bar v_0>0 \}\cap O$,} $$
and the equation is satisfied in the classical sense. We find $$\xi_0 \notin \{ u^*>0 \} \, \, \subset \{ \bar v_0 >0 \},$$ since in $\{u^*>0\}$
both $u^*$ and $\bar v_0$ solve the same equation. On the other hand if $\xi_0 \in \{u^*=0\}$ then (see \eqref{59}) we obtain $\nabla \bar v_0(\xi_0)=0$ thus
$$\bar v_0 \ge \bar v_0 (\xi_0) >0,$$
and we reach again a contradiction.

\qed

Next we prove the induction step for Theorem \ref{T3}.

 \begin{prop}\label{p5}
 Assume Theorem \ref{T3} holds for all $l \le k-1$ for some $k$ with $1 \le k \le n-2$. Then Theorem \ref{T3} holds also for $k$.
 \end{prop}

We recall the notation of Section 3 that we denote points in $\R^n$ by
$$x=(y,z,x_n)  \quad y=(x_1,..,x_k) \in \R^k \quad z=(x_{k+1},..,x_{n-1}) \in \R^{n-1-k}.$$

 The proof of Proposition \ref{p5} is very similar to the proof of Proposition \ref{p4}, in most statements we just have to replace $x'$ by $z$. We provide the details below.

 {\it Remark:} In view of Lemma \ref{l5} we may assume, after relabeling $\mu$, that all renormalized solutions $u_h$ given in Lemma \ref{l5} are in the same class $\mathcal D_0^\mu(1,..,1, \infty,...,\infty).$

 Let $\Gamma_u$ denote the tangent cone of $u$ at the origin. Any supporting plane for $u$ at the origin has $0$ slope in the $y$ direction, hence $\Gamma_u$ does not depend on the $y$ variable.

We define the $n-k-1$ dimensional function $\gamma_u(z)$ as being the restriction of $\Gamma_u$ to $x_n=1$ i.e
$$\gamma_u(z):= \Gamma_u(y,z,1).$$
By construction the upper graph of $\gamma_u$ is a closed set and
$$u(x) \ge \Gamma_u(x)=x_n \gamma_u(\frac{z}{x_n}).$$
Since $\nabla u(0)=0$ we have $\gamma_0 \ge 0$ and $\inf \gamma_u=0$. In the next lemma we obtain some useful properties of $\gamma_u$.

\begin{lem}\label{l9}

a) $$\gamma_u(z) \ge c_0|z|-C_0.$$

b) If $$\gamma_u \ge C_0 p_z \cdot (z-z_0),$$for some unit vector $p_z\in \R^{n-k-1}$, $|p_z|=1$ and some $z_0$ then
$$\gamma_u \ge p_z \cdot (z-z_0) + c_0.$$
The constants $c_0$, $C_0$ above are universal constants.
\end{lem}

\begin{proof}
a) We compare $u$ with
$$w:=c z \cdot p_z + c|z|^2 + \psi_u(y) + C (x_n^2 -\mu^{-1} x_n),$$
with $p_z$ a unit vector, and $\psi_u$ denoting the boundary data of $u$ on $\p \Om \cap\{(y,0,0) \}$.

Notice that $w = u$ on the intersection of $\p \Om$ with the $y$ axis. We choose $c$ small such that $w \le 1$ in $\Om \subset B^+_{1/\mu}$ and $C$ large such that $\det D^2 w \ge \det D^2u$. By maximum principle, $u \ge w$, hence
$$\gamma_u \ge \gamma_w=c z \cdot p_z -C x_n,$$
which proves part a).

b) Assume that $p_z$ points in the $z_1$ direction and let $z_0 \cdot p_z=q$. Then
$$u \ge x_n\gamma_u(\frac{z'}{x_n}) \ge C_0(z_1-q x_n)^+.$$
In the set $O=\Omega \cap \{z_1-q x_n >-1\}$ we compare $u$ with
$$w:= \frac {C_0}{2} (z_1-qx_n) + \frac{C_0}{8} (z_1-q x_n)^2 + \delta (|x|^2-z_1^2) + \delta x_n,$$
where $\delta$ is small, fixed, depending on $\mu$.

Notice that if $C_0$ is sufficiently large we have $\det D^2 w \ge \det D^2 u$ and $u \le w$ on $\p O$. Indeed,
on $\p O \setminus \p \Om$ we have $$z_1-q x_n=-1 \quad \Rightarrow \quad w \le 0 \le u,$$
and on $\p O \cap \p \Omega$
$$1 \ge u \ge C_0(z_1 - q x_n) \quad \Rightarrow \quad w \le 1.$$
From the inequalities above and $$w(y,0,0) \le \delta |y|^2 \le u(y,0,0)$$ we obtain $w \le u$ on $\p O$.
 In conclusion $$\gamma_u \ge \gamma_w \ge \frac{C_0}{2}(z_1-q) + \delta.$$

\end{proof}

\begin{rem}\label{r6}
In part a) we showed that
\begin{equation}\label{59.6}
u(x) \ge \psi_u(y) + c|z|- Cx_n.
\end{equation}
Also we only need the weaker assumption $$u \ge C_0(z_1-q x_n) \quad \mbox{ on $\p \Omega$,}$$
in order to obtain the conclusion of part b).
\end{rem}

Part a) shows that the point $z_o \in \R^{n-k-1}$ where $\gamma_u$ achieves its infimum belongs to $B_C^z$. As a consequence of Lemma \ref{l9} we obtain as before the following inclusions
$$B_{c_*}^z \subset S_1(\gamma_u) - z_o \subset B_{C_*}^z$$
$$S_{c_*}(\gamma_u)-z_o \subset (1-c_*) (S_1(\gamma_u)-z_o).$$
for universal constants $c_*$ small, $C_*$ large.

Next we write these inclusions for the rescalings $u_h$ of $u$ defined in Lemma \ref{l6} (see Remark \ref{r4}). Recall from Lemma \ref{l5} that for any $h \in (0,1]$, we have
$$u_h(x)=\frac 1 h \tilde u(D_h x), \quad \quad \mbox{with} \quad \tilde u(x)=u(T_h A_h x).$$
Notice that $\gamma_u$ is just a translation of $\gamma_{\tilde u}$.
Since
$$\Gamma_{u_h}(x)=\frac 1 h \Gamma_{\tilde u}(D_h x)$$
we divide by $x_n$ and obtain as before
$$\gamma_{u_h}(z)=\frac{d_n}{h} \, \, \gamma_{\tilde u}\, (d_n^{-1} \, D_h^z z) \quad \quad \mbox{with} \quad D_h^z=diag(d_{k+1},..,d_{n-1}).$$

This implies that $$d_n^{-1}D_h^z\, \,  S_s(\gamma_{u_h})= S_{s \, h/d_n}(\gamma_{\tilde u}).$$

We obtain the following inclusions for the sections of $\gamma_u$,
\begin{equation}\label{59.7}
d_n^{-1} D_h^z \, B_{c_*}^z \quad \subset \quad S_t(\gamma_u) -z_o \quad  \subset \quad d_n^{-1} D_h^z \, B_{C_*}^z, \quad  \quad \quad t=\frac{h}{d_n},
\end{equation}
and
\begin{equation*}
S_{c_* t}(\gamma_u)-z_o \quad \subset \quad (1-c_*)\left(S_t(\gamma_u) -z_o   \right).
\end{equation*}

From Lemma \ref{l5} we know that as $h$ ranges from 1 to 0 the parameter $t=h/d_n$ covers an interval $[0,c]$. The inclusions above show that the sections $S_t(\gamma_u)$ are balanced around the minimum point $x_o$, and
 \begin{equation}\label{510}
\gamma_u(z) \ge c |z-z_o|^M \quad \quad \mbox{in $B_c^z(z_o)$,}
\end{equation}
for some $M$ large universal.
We also recall from Lemma \ref{l5} that, from the construction of $D_h$,
\begin{equation}\label{511}
 \mathcal S_h'(u) \, \, \subset \, \,\R^k \times h^{-1/2} D_h^z \, B_{C(n)}^z.
 \end{equation}

It remains to show that $\gamma_u$ grows quadratically near its minimum point.

\begin{lem}\label{l10}
There exists universal constants $c_1$, $C_1$ such that
$$c_1|z-z_o|^2 \le \gamma_u(x') \le C_1 |z-z_o|^2 \quad \quad \mbox{in $B_{c_1}^z(z_o)$}.$$
\end{lem}

This lemma implies Proposition \ref{p5} as before. Indeed, the lemma gives $$t^{1/2}B_c^z \subset S_t(\gamma_u) -z_o \subset t^{1/2} B_C^z,$$
which together with \eqref{59.7} implies that
$$ c d^{1/2}_n \le d_i h^{-1/2} \le C d^{1/2}_n \quad \quad \mbox{for $k<i<n$.}$$
By Lemma \ref{l5} we also know
$$ c \le d_i h^{-\frac 12} \le C \quad \quad \mbox{for $i\le k$.}$$
Then
\begin{equation}\label{65.65}
\left(\prod_{i=1}^{n-1} d_i^2 \right) \, d_n^{2+\alpha}=h^n \quad \Rightarrow \quad ch \le {d_n}^{n+1-k+\alpha} \le Ch,
\end{equation}
and Proposition \ref{p5} follows from \eqref{511}.

\

Below we prove Lemma \ref{l10}. After performing a sliding along $x_n=0$ of bounded norm, we may assume that $z_o=0$.

\

{\it Step 1:} We use Theorem \ref{Po} in the set $\{ u < c x_n\}$ to obtain
\begin{equation}\label{512}
D^2\gamma_u \le C I \quad \quad \mbox{in} \quad \{\gamma_u <c \}.
\end{equation}

We first need to bound $|\nabla u|$ in the set $\{u<c x_n\}$ for some $c$ small. To this aim we observe that the orthogonal projection of $\Om=S_1(u)$ into the $z$-axis contains the ball $B_{1/C_0}^z$, with $C_0$ as in Lemma \ref{l9}. Otherwise we can find a direction, say $z_1$ such that
$$\Om \subset \{ z_1 \le 1/C_0 \},$$
hence $$u \ge C_0 z_1 \quad \mbox{on $\p \Om$},$$
and by Lemma \ref{l9} (see Remark \ref{r6}) we find $$\gamma_u \ge z_1 + c_0,$$ and we contradict that $\gamma_u(0)=0$ (since $z_o=0$.)

 Notice that $ \Omega \subset B_C^+$ contains a ball $B_\mu(x^*)$, the projection of $\Om$ into the $z$ coordinates contains $B_{1/C_0}^z$ and also $$G:=\{(y,0,0)| |y| \le c \} \subset \p \Om.$$ Since $\Om$ is convex, it must contain also $B_r(\varrho e_n)$ for some small fixed universal constants $\varrho$, $r$. Now we use that at each point in $G$ the function $u$ has a supporting plane of bounded slope, and conclude that $|\nabla u| \le C$ in the convex set generated by $G$ and $B_{r/2}(\varrho e_n)$. This convex set contains $\{ u < cx_n \}$ if $c$ is sufficiently small, since by \eqref{31}, \eqref{59.6} and \eqref{510} we obtain
$$\{ u< \delta x_n \} \subset \{x_n \le c(\delta)\} \cap \{|z| \le c(\delta) x_n \} \cap \{\psi_u \le c(\delta) \},$$
for some constant $c(\delta)\to 0$ as $\delta \to 0$.
 Now let $w:=u-cx_n$ and notice that the rescalings $w_\lambda$ defined in ...
 have uniform gradient bound in $\{w_\lambda <0\}$ and they converge uniformly on $x_n=1$ to $|\Gamma_u - c|$. Step 1 follows by applying Theorem \ref{Po} to $w_\lambda$ as before.

\

Above we showed also that the segment $[0,\varrho Z_o] \subset S_1(u)$ with $Z_o:=(0,z_o,1)=e_n$. We apply this for the rescaling $u_h$ and obtain
$$[0,\varrho d_n  \, Z_o] \subset S_h(u).$$ Using the bounds on $d_n$ from Lemma \ref{l5} we find (see also \eqref{31})
\begin{equation}\label{513}
 c t^{1/\delta} \le  u(0,tz_o ,t) \le C t^{4/3}.
 \end{equation}
We can extend this inequality at points $z$ near $z_o$,
\begin{equation}\label{514}
c t^{1/\delta} \le u(0,t z,t) - t \gamma_u(z) \le C t^{4/3} \quad \quad \mbox{for all $z \in B_c^z$.}
\end{equation}
Indeed, \eqref{514} follows by applying \eqref{513} to the function $u-l=u-p_z \cdot z-p_n x_n$ where $p_z \cdot z +p_n$ is the linear function tangent by below to $\gamma_u$ at some point $z_*$. From Step 1 we see that when $|z_*|$ is small, $|p_z|$, $|p_n|$ are also small and $u-l$ (renormalized at its $1/2$ section) belongs to a class $\mathcal D_0^c (1,..,1,\infty,..,\infty)$. Therefore we can apply \eqref{513} for $u-l$ and obtain the desired inequality \eqref{514} for $z_*$.

\

{\it Step 2:} In step 2 we apply Theorem \ref{abo} for the Legendre transform of $u$ and obtain
$$D^2 \gamma_u \ge c I \quad \quad \mbox{in $B_c^z$.}$$

 As before let $u^*$ denote the Legendre transform of $u$,
 $$u^*(\xi):= \sup_{x \in \ov \Om} \left( x \cdot \xi - u(x)\right).$$
 Writing $\xi=(\xi_y,\xi_z,\xi_n)$ we have
 $$u^*(\xi) \ge \sup_{(y,0,0) \in \ov \Om} \left( y \cdot \xi_y - u(y,0,0)\right) = \psi_u^*(\xi_y),$$
 where $\psi_u^*$ is the Legendre transform of the boundary data $\psi_u(y)$ of $u$ (on the $y$ subspace).
  If $|\xi|$ is small then the maximum in $u^*(\xi)$ is realized either in $\Omega$ or at some point $(y,0,0)$ with $|y|$ small and in the second case we have $u^*=\psi_u^*$.
  Since $u$ is strictly convex in $\Om$ we find that
  $$u^* \in C^1(B_c) \quad \mbox{and} \quad \quad (u^*_n)^\alpha \det D^2 u^* =1 \quad \quad \mbox{in $\{u^*>\psi_u^*\}$.}  $$

  In other words $u^*$ is the solution to an obstacle problem in which the obstacle $\psi_u^*$ is quadratic and depends only on the $\xi_y$ variable.

We define $K$ as the convex set
$$K:=\{ u^*=0\} \, = \, \{\xi_y=0  \} \cap \{ \xi_n \le - \gamma_u^*(\xi_z)\}$$
where $\gamma_u^*$ denotes the Legendre transform of $\gamma_u$. Below we bound the curvatures of the level sets of $u^*$ in the $z$ direction in a neighborhood of the origin. Then Step 2 follows by applying these bounds for the $0$ level set above.

From Step 1 and \eqref{510} we know that
$$c|z|^M \le \gamma_u(z) \le C |z|^2  \quad \mbox{in $B_c^z$}$$
hence
\begin{equation*}
\{ \xi_n \le -C |\xi_z|^\frac{M}{M-1} \}  \subset K \subset \{ \xi_n  \le - c |\xi_z|^2 \} \quad \mbox{in $\{\xi_y=0  \} \cap B_c$}.
\end{equation*}

Also from \eqref{514} and the definition of Legendre transform we find
$$ u^*(\xi) \ge c \left ( (\xi_n + \gamma_u^*(\xi_z))^+ \right )^4 \quad \quad \mbox{in $B_c$},$$
which together with
$$u^*(\xi) \ge \psi_u^*(\xi_y) \ge c |\xi_y|^2,$$
implies that
$$O:= \{ u^* < \eta (\xi_n + \eta) \} \subset B_{c_1},$$
with $\eta$ and $c_1$ sufficiently small universal constants.

We also claim that in $B_c$, $|\nabla u^*|$ and the Lipschitz norms in the $z$ direction of the level sets of $u^*$ (viewed as graphs in the $-e_n$ direction) are bounded by a universal constant.
Indeed, let $\xi \in B_c$ and let $\nabla u^*(\xi)=x=(y,z,x_n) \in \Omega$. We need to show that
$|z| \le C x_n$. We increase the tangent plane of $u$ at the point $x$ (which has slope $\xi$) till it touches the boundary data of $u$ for the first time at some point $(y_0,0,0)$. Clearly $\xi_y$ coincides with the derivative of $\psi_u$ at $y_0$. We have
\begin{align*}
u(x) & \le \psi_u(y_0) + \xi \cdot (x-(y_0,0,0)) \\
& \le \psi_u(y_0) + \xi_y \cdot (y-y_0) + \xi_z \cdot z + \xi_n x_n  \\
& \le \psi_u(y) + \xi_z \cdot z + \xi_n x_n,
\end{align*}
and by \eqref{59.6}
$$u(x) \ge \psi_u(y) + c|z| - Cx_n.$$
The inequalities above imply that $|z| \le C x_n $ if $|\xi|$ is sufficiently small.

Since $u^*$ is not strictly increasing in the $e_n$ direction, we apply Theorem \ref{abo} using approximations as before.
Formally, $u^*$ satisfies the hypotheses of Theorem \ref{abo} with right hand side $f(u^*-\psi_u^*)$ with $f=\chi_{(0,\infty)}$. The right hand side does not depend on the $z$ variable, thus we can bound the second derivatives of the level sets of $u^*$ in the $z$ direction.

\

{\it Approximation.} Define $$v_\eps= \max \left \{ u^*, \, \, \psi_u^*(\xi_y)+ \eps \left (1+ \xi_n + \frac 1 2 |\xi|^2\right ) \right \}$$
and we remark that in $B_{c_1}$, $v_\eps$ is strictly increasing in the $e_n$ direction, $|\nabla v_\eps| \le C$, and its level sets have Lipschitz norm in the $z$ direction bounded by a universal constant. In the set $\{v_\eps>u^*\} \cap B_{c_1}$ we have
$$(\p_n v_\eps)^ \alpha \det D^2 v_\eps \ge \eps^{n+\alpha},$$
hence
$$ (\p_n v_\eps)^ \alpha \det D^2 v_\eps \ge f_\eps(v_\eps-\psi_u^*)  \quad \mbox{in $B_{c_1}$,} $$
in viscosity sense, with
$f_\eps$ a nondecreasing function satisfying $$f(s)=\eps^{n+\alpha} \quad \mbox{if $s \le \eps^{1/2}$}, \quad \quad f(s)=1 \quad \mbox{if $s \ge 2 \eps^{1/2}$.} $$

 We define $\bar v_\eps$ as the viscosity solution to
 $$(\p_n \bar v_\eps)^\alpha \det D^2 \bar v_\eps =f_\eps(\bar v_\eps-\psi_u^*) \quad \quad \mbox{in}\quad O_\eps:= \{v_\eps< \eta(\xi_n + \eta)\},$$
 $$ \bar v_\eps=v_\eps \quad \mbox{on $\p O_\eps$.} $$

We apply Theorem \ref{abo} for $\bar v_\eps$ in $O_\eps$ and obtain the uniform second derivative bounds in the $z$ direction for its level sets around the origin. As before we find that $\bar v_\eps$ converges to $u^*$ as $\eps \to 0$. Thus the conclusion holds also for $u^*$, and the proof of Proposition \ref{p5} is finished.

\qed

\section{Proof of Theorem \ref{T2}}\label{s7}

Assume the hypotheses of Theorem \ref{T2} are satisfied. By Theorem \ref{T1}, we may also assume after performing an affine transformation that the solution $u$ satisfies
\begin{equation}\label{60}
c_0(|x'|^2 + x_n^{2+\alpha}) \le u(x) \le C_0 (|x'|^2 + x_n^{2+\alpha})
\end{equation}
for all $|x| \le c(\rho,\rho')$. Since in our case $\mu=1$, the constants $c_0$, $C_0$ depend only on $\alpha$ and $n$.

The rescalings $u_h$ for small $h$,
\begin{equation}\label{60.1}
u_h(x):=\frac 1 h u \left (h^\frac 12 x', h^\frac{1}{2+\alpha} x_n \right ),
\end{equation}
satisfy inequality \eqref{60} as well, and therefore belong to a compact family. Precisely, given a sequence $u^m$ of functions as above, and $h_m \to 0$, we can extract a subsequence $u_{h_m}^m$  that converges uniformly on compact sets to a global solution $u_0$ defined in $\R^n_+$, that satisfies \eqref{60} and
$$ (1-\eps_0) x_n^\alpha \le \det D^2 u_0 \le (1+\eps_0) x_n^\alpha,$$
$$ (1-\eps_0)\frac{|x'|^2}{2} \le u_0(x',0) \le (1+\eps_0) \frac{|x'|^2}{2}.$$

By compactness, the proof of Theorem \ref{T2} follows from the following Liouville type theorem.

\begin{prop}\label{p6}
Let $u \in C(\ov{\R^n_+})$ be a convex function that satisfies the growth condition
\begin{equation}\label{60.5}
c_0(|x'|^2 + x_n^{2+\alpha}) \le u(x) \le C_0 (|x'|^2 + x_n^\alpha),
\end{equation}
with $c_0$, $C_0$ the constants of Theorem \ref{T1} (see Remark \ref{r0}) and
\begin{equation}\label{61}
\det D^2 u = x_n^\alpha, \quad \quad  u(x',0) = \frac 12 \, |x'|^2.
\end{equation}
Then $$u=U_0:=\frac 12 |x'|^2 + \frac{x_n^{2+\alpha}}{(1+\alpha)(2+\alpha)}.$$
\end{prop}

In the case $\alpha=0$ the conclusion is slightly different and $u$ must be a quadratic polynomial. The proof follows from the Pogorelov estimate in half space (see \cite{S1}). When $\alpha>0$ the situation is more delicate and we will make use of Theorem \ref{T1}. 

We define $\mathcal K$ as the set of functions $u \in C(\ov{\R^n_+})$ that satisfy \eqref{60.5}, \eqref{61}. We want to show that $\mathcal K$ consists only of $U_0$.

Clearly $\mathcal K$ is a compact family under uniform convergence on compact sets. Also for any $h>0$,
$$ u \in \mathcal K \quad \Rightarrow \quad u_h(x):= \frac 1 h u \left (h^\frac 12 x', h^\frac{1}{2+\alpha} x_n \right ) \in \mathcal K.$$

If $u \in \mathcal K$ and $$x_0 \in \{x_n=0\}$$ is a point on the boundary then, after subtracting its tangent plane at $x_0$ and performing an appropriate sliding $A_{x_0}$, we can normalize $u$ at $x_0$ such that it belongs to $\mathcal K$. Precisely, there exists $A_{x_0}$ sliding along $x_n=0$ such that $u_{x_0} \in \mathcal K$ where
$$u_{x_0}(x):= u(x_0+A_{x_0}x)- u(x_0)-\nabla u(x_0) \cdot A_{x_0}x.$$

This statement follows from Theorem \ref{T1}. If $x_0$ is sufficiently close to the origin then the tangent plane of $u$ at $x_0$ has bounded slope. Indeed, the upper bound for $u_n(x_0)$ is obtained from \eqref{60.5} by convexity while for the lower bound we compare $u$ in $S_1(u)$ with an explicit barrier of the type
$$-\frac 12 x_0^2 + x' \cdot x_0 + \frac c2 |x'-x_0|^2 + c^{1-n} (x_n^2-Mx_n),$$
with $c$ small and $M$ large appropriate constants.

We can apply Theorem \ref{T1} at the point $x_0$ in the section $S_1(x_0)$ of $u$ and find that $u_{x_0}$ defined above satisfies \eqref{60.5} in a fixed neighborhood around the origin. In the general case we apply this argument for $u_h$ with $h \to \infty$ and obtain that $u_{x_0}$ satisfies \eqref{60.5} in whole $\R^n_+$.

Below we provide the proof of Proposition \ref{p6} in several steps. The main ingredients are the compactness of the class $\mathcal K$ under the rescalings and normalizations given above, the fact that for $i<n$, $u_{ii}$ are subsolutions for the linearized operator and also that $$\frac{u_n}{x_n^{1+\alpha}}$$
solves an elliptic equation.

\

{\it Step 1:} We show that if $u \in \mathcal K$ then $D^2_{x'} u \le I$.

Given any $y_0 \in \R^n_+$ we consider the section $S_h(y_0)$ of $u$ that becomes tangent to $x_n=0$ at some point $x_0$. After normalizing $u$ at $x_0$ and then after an appropriate rescaling, we may assume that $S_h(y_0)=\{ u < x_n \}$. Notice that the tangential second derivatives $D^2_{x'}u$ are left invariant by these transformations. Hence, by interior regularity, $ u_{ii} \le C $ for $i<n$. Assume we have a sequence of functions $u_m \in \mathcal K$ and points $y_m$ (normalized as above) for which $\p_{ii} u_m(y_m)$ tends to the supremum value $ \sup_{u \in K}\p_{ii} u$. Then we may assume that $u_m \to \bar u \in \mathcal K$, and $\p_{ii} \bar u$ achieves an interior maximum at the point $\bar y=\min (\bar u-x_n)$. The function $\p_{ii} \bar u$ is a subsolution for the linearized operator, thus $\p_{ii} \bar u$ is constant in $\R^n$. The boundary data of $\bar u$ on $x_n=0$ shows that this constant must be 1, and this proves Step 1.

\

{\it Step 2:} We show that if $u \in \mathcal K$ then $$\psi(x'):=u_n(x',0) \quad \mbox{is concave,} \quad \psi \le 0 \quad \mbox{and}\quad \|\nabla \psi \|_{C^\frac{\alpha}{2+\alpha}(\R^{n-1})} \le C.$$

Formally, by Step 1 we have $u_{iin}\le 0$ on $x_n=0$ hence $\psi =u_n$ is concave. We prove this rigourously below. Let $x_0=(x_0', h^\frac{1}{2+\alpha})$ be the point where the section $S_t$ at the origin (for some $t$) becomes tangent to $x_n=h^\frac{1}{2+\alpha}$. From \eqref{60.5} we have
$$ch \le t \le Ch, \quad \quad |x_0'| \le C t^\frac 12 \le C h^ \frac 12.$$
We use Step 1 and $\nabla_{x'}u(x_0)=0$ and obtain
$$u(x',h^\frac {1}{2+\alpha}) \le C h + \frac 12 |x'-x_0'|^2 \le Ch + C |x'| h ^\frac 12 + u(x',0).$$
We let $h \to 0$ and obtain $u_n(x',0) \le 0$.

If $|x'| \le 1$ then as above, we use an explicit barrier for $u$ and easily obtain also a lower bound $0 \ge u_n(x',0) \ge -C$. We apply this for
the rescaling $u_h$ (see \eqref{60.1}) and find
$$ 0 \ge \p_n u_h(x',0)=h^{\frac{1}{2+\alpha}-1} u_n(h^\frac 12 x',0) \ge -C \quad \quad \mbox{if $|x'| \le 1$,}$$
hence
$$0 \ge u_n(x',0) \ge - C |x'|^{1+\frac{\alpha}{2+\alpha}} \quad \quad \mbox{for all $x'$.}$$

We apply this last inequality for $u_{x_0}$, the normalization of $u$ at $x_0\in \{x_n=0\}$,
$$u_{x_0}(x)=u(x_0+A_{x_0}x)-u(x_0)-\nabla u(x_0) \cdot A_{x_0}x,$$
with $$A_{x_0}x=x-\tau_{x_0} x_n, \quad \quad \tau_{x_0} \cdot e_n=0.$$
We find
$$0 \ge \p_n u_{x_0} \, \, (x',0)=u_n(x_0+x',0)-u_n(x_0)-x' \cdot \tau_{x_0} \ge - C |x'|^{1+\frac{\alpha}{2+\alpha}},$$
where in the equality above we made use of $\nabla_{x'}u(x',0)=x'$. This proves Step 2 and we remark that the inequality above shows that the components of the vector $\tau_{x_0}$ in the sliding $A_{x_0}$ are given by $u_{ni}(x_0)$, $i<n$.

\

{\it Step 3.} We show that if $v$ is a convex function in $\R^n_+$ that satisfies
$\det D^2v=x_n^\alpha,$
then
$$\bar w:=v_n/ x_n^{1+\alpha}$$
satisfies in the set $\{ \bar w>0 \}$ a linear elliptic equation of the type
$$ L \bar w:=a^{ij}(x) \bar w_{ij} + b^i(x) \bar w_i =0, \quad \quad \mbox{with} \quad (a_{ij}(x))_{ i,j} >0.$$

It suffices to show that $$w=\log \bar w = \log v - (1+\alpha) \log x_n,$$
satisfies a linear elliptic equation as above. We have
\begin{align*}
w_i&=\frac{v_{ni}}{v_n}-\frac{1+\alpha}{x_n}\delta_n^i  \\
w_{ij}&= \frac{v_{nij}}{v_n}-\frac{v_{ni}v_{nj}}{v_n^2}+\frac{1+\alpha}{x_n^2}\delta_n^i \delta_n^j.
\end{align*}
Differentiating the equation $\log \det D^2 v=\alpha \log x_n$ along $x_n$ direction
$$v^{ij}v_{nij}=\frac{\alpha}{x_n},$$
hence
\begin{equation}\label{62}
v^{ij} w_{ij} =  \frac{\alpha}{x_n v_n}-\frac{v_{nn}}{v_n^2}+\frac{1+\alpha}{x_n^2}v^{nn}= -\frac{w_n}{v_n}-\frac{1}{x_nv_n}+\frac{1+\alpha}{x_n^2}v^{nn}.
\end{equation}

We have $$v^{nn}v_{nn}=1-\sum_{i \ne n} v_{ni}v^{in}=1-\sum_{i \ne n} w_i v_n v^{in},$$
hence
$$v^{nn}=\frac{1}{v_{nn}}- g^i w_i,$$
for some functions $g^i$. Then
\begin{align*}
\frac{1+\alpha}{x_n}v^{nn}- \frac{1}{v_n}&=\frac{1+\alpha}{x_nv_{nn}}-\frac {1}{v_n}-\tilde g^i w_i\\
&=-\frac{1}{v_{nn}} \left( \frac{v_{nn}}{v_n}-\frac{1+\alpha}{x_n} \right)-\tilde g^i w_i\\
&=-\frac{1}{v_{nn}}w_n-\tilde g^i w_i,
\end{align*}
which together with \eqref{62} proves step 3.

\

As a consequence we obtain that if $v \in C(\ov \R^n_+)$ is a convex function that satisfies
\begin{equation}\label{63}
\det D^2 v=x_n^\alpha, \quad \quad v(x',0)=\frac 12 |x'|^2,
\end{equation}
and $\bar w=v_n/x_n^{1+\alpha}$ achieves a positive maximum at an interior point then $\bar w$ is constant. This implies that $v=U_0$ with $U_0$ as in Proposition \ref{p6} and therefore $\bar w \equiv \frac{1}{1+\alpha}$.

Assume that $v$ satisfies \eqref{63}, and for some direction $\xi=(\xi',1)$ and constant $m$, the function
$$\tilde w = \frac{v_\xi+m}{x_n^{1+\alpha}} \quad \mbox{has a positive interior maximum.}$$
Then $\tilde w \equiv \frac{1}{1+\alpha}$. Indeed, the function
$$\tilde v(y):=v(y'+\xi'y_n,y_n)+m y_n,$$
satisfies \eqref{63} and the conclusion follows as above since
$$\frac{\tilde v_n(y)}{y_n^{1+\alpha}}=\frac{v_\xi(x)+m}{x_n^{1+\alpha}}, \quad \quad \quad x:=(y'+\xi' y_n,y_n).$$

\

{\it Step 4.} We use the result above and show that $$u \in \mathcal K \quad \quad \Rightarrow \quad \frac{u_n}{x_n^{1+\alpha}} \le \frac{1}{1+\alpha}.$$

Let $x^*$ be a point in $\R^n_+$ where $u_n(x^*)>0$, and let $x_0$ be the point where the first section of $u$ at $x^*$ becomes tangent to $x_n=0$. As in Step 1 we normalize $u$ at $x_0$ and then rescale
$$v(y):=  \frac 1 h u_{x_0}(F_hy)  = \frac 1 h \left [u(x_0+A_{x_0}F_hy)-u(x_0)-\nabla u(x_0) \cdot A_{x_0}F_h y\right],$$
with
$$F_h y:=(h^\frac 12 y',h^\frac{1}{2+\alpha}y_n), \quad \quad A_{x_0}x=x - \tau_{x_0} x_n, \quad \quad \quad \tau_{x_0}=(\tau^1,...,\tau^{n-1},0).$$

We know that $v \in \mathcal K$ and we denote by $y^*$ the corresponding coordinates for $x^*$ in the $y$ coordinate,
$$x=x_0 + A_{x_0}F_h y \quad \quad x^*=x_0+A_{x_0}y^*.$$
We choose $h$ above such that $y^*$ is the center of the section $\{ v<y_n\}$, i.e. the point where $v-y_n$ achieves its minimum. We have
$$h \nabla v=\nabla u \, A_{x_0}F_h - \nabla u(x_0) \, A_{x_0} F_h,$$
and we obtain
\begin{equation}\label{64}
u_n=u_n(x_0) + h^\frac{1+\alpha}{2+\alpha}v_n + h^\frac 12 \tau^i v_i
\end{equation}
where $u_n$ is evaluated at $x$ and $v_n$, $v_i$ are evaluated at $y$.

 Since $u_n(x^*)>0$ and $|\nabla v(y^*)| \le C_1$ for some constant $C_1$ depending only on $\alpha$ and $n$ we find
\begin{equation}\label{65}
0 \le u_n(x_0) + C_1 \left (h^\frac{1+\alpha}{2+\alpha} + h^\frac 12 |\tau_{x_0}| \right).
\end{equation}
 On the other hand by Step 2 we know that $u_n \le 0$ on $x_n=0$. Thus if we write \eqref{64} at $$y=(y',0), \quad \mbox{with} \quad y'=2C_1 \frac{\tau_{x_0}}{|\tau_{x_0}|},$$
and use $\nabla_{y'}v(y)=y'$ together with \eqref{65} we obtain
$$ 0 \ge - C_1 \left (h^\frac{1+\alpha}{2+\alpha} + h^\frac 12 |\tau_{x_0}| \right) - C_2 h^\frac{1+\alpha}{2+\alpha} + 2 C_1 h^\frac 12 |\tau_{x_0}|$$
for some $C_2$ large depending on $C_1$. This and \eqref{65} show that
$$|\tau_{x_0}| h^\frac 12 \le C_3 h^\frac{1+\alpha}{2+\alpha}, \quad \quad u_n(x_0) \ge -C_3 h^\frac{1+\alpha}{2+\alpha},$$
for some $C_3$ depending only on $n$ and $\alpha$. We use these inequalities in \eqref{64} and obtain
\begin{equation}\label{66}
\frac{u_n(x)}{x_n^{1+\alpha}}=\frac{m+v_\xi(y)}{y_n^{1+\alpha}},
\end{equation}
for some vector $\xi$ and constant $m$ satisfying
$$\xi=(\xi',1), \quad \quad \quad |\xi'|\le C_3, \quad \mbox{and} \quad -C_3 \le m \le 0.$$

The right hand side of \eqref{66} is bounded by a universal constant at $y^*$ which implies that $u_n/x_n^{1+\alpha}$ is bounded at $x^*$. Since $x^*$ is arbitrary we obtain an upper bound for this function. Moreover, if we take a sequence of points which approach its supremum then the corresponding functions $v$ (and $m$, $\xi$) converge up to a subsequence to a limiting solution $\bar v \in \mathcal K$ (respectively $\bar m$, $\bar \xi$) for which
$$ \frac{ \bar m+\bar v_{\bar \xi}}{z_n^{1+\alpha}} \quad \mbox{achieves its maximum at the center of $\{\bar v <y_n \}$.}$$
By Step 3 we obtain that this maximum value is $1/(1+\alpha)$.

\

{\it Step 5.} We show that if $u \in \mathcal K$ then $u=U_0$.

Indeed, we integrate in the $x_n$ direction the inequality in Step 4 and obtain $u \le U_0$. Assume by contradiction that $u$ does not coincide with $U_0$ hence, by strong maximum principle, $u < U_0$ in $\R^n_+$. Let
$$V:= \frac{1+\eps}{2} |x'|^2 +\frac{(1+\eps)^{1-n}}{(2+\alpha)(1+\alpha)}x_n^{2+\alpha} - \eps x_n,$$ and notice that $\det D^2 V= \det D^2 u$, and
$$V \ge U_0 \ge u \quad \quad \mbox{ on} \quad  \{x_n=0\} \cup \left(\{|x'| = C_1\} \cap \{0 \le x_n \le 1\} \right)$$
and if $\eps$ is sufficiently small
$$V \ge U_0 - C_2 \eps \ge u  \quad \quad \mbox{on} \quad \{|x'| \le C_1\} \cap \{x_n = 1\},$$
where $C_1$, $C_2$ are constants depending on $\alpha$ and $n$. By maximum principle $$V \ge u \quad \mbox{in} \quad B_{C_1}' \times [0,1]$$
and we contradict $\nabla u(0)=0$ which follows from \eqref{60.5}.

\qed

\section{Consequences of Theorem \ref{T2}}\label{s8}

In this section we use Theorem \ref{T2} and prove Theorems \ref{T2.1}, \ref{T03}, \ref{T02}.
First we show that if the hypotheses of Theorem \ref{T1} or Theorem \ref{T2} are satisfied at a point then they hold also in a neighborhood of that point.

\begin{lem}\label{l11}
Assume the hypotheses H1, H2, H3, H4 of the localization Theorem \ref{T1} are satisfied and, in addition, $\p \Om$ admits an interior tangent ball of radius $\rho$ at all points on $\p \Om \cap B_\rho$ and
$$u(x)=\ph(x') \quad \mbox{on} \quad \p \Om \cap B_\rho, \quad \quad \mu^{-1} I \ge D_{x'}^2 \ph \ge \mu I.$$
Then the hypotheses of the localization theorem hold at all points $x_0 \in \p \Om \cap B_c$, for some $c=c(\rho,\rho')$ small.
\end{lem}

\begin{proof}
  We only have to check that on $\p \Om$, $u$ separates quadratically away from the tangent plane at $x_0$, hence we need to show that $|\nabla u(x_0)|$ is sufficiently small when $|x_0|$ is close to the origin. By Theorem \ref{T1} there exists a sliding $A$, $|A|\le C_1(\rho,\rho')$ such that
 for $h \le c_1(\rho,\rho')$ small, the rescaled function
 \begin{equation}\label{80}
 u_h(y):=\frac 1 h u(A F_h y) \quad \quad F_h y:=(h^\frac 12y',h^\frac{1}{2+\alpha}y_n), \quad x=AF_h y,
 \end{equation}
satisfies in $S_1(u_h)$
$$ u_h(y)= \ph_h(y') \quad \mbox{on} \quad \p \Om_h=(A F_h)^{-1} \p \Om, \quad \quad \quad \frac \mu 2 I \le D^2_{y'} \ph_h \le 2 \mu I,$$
$$c_0 (|y'|^2+y_n^{2+\alpha}) \le u_h \le C_0(|y'|^2+y_n^{2+\alpha}), \quad \quad \det D^2u_h \le 2 y_n^\alpha.$$
where the last inequality follows from the fact that $u$ satisfies the same inequality in $\Om \cap B_\rho$. Now, if $y_0 \in \p \Om_h$ with $|y_0|<c$ small we can bound $|\nabla u_h(y_0)|$ as in Section \ref{s7}, by using a lower barrier of the type
$$u_h(y_0)+\xi' \cdot z'+c|z'|^2+c^{1-n} (z_n^2-Mz_n),$$
where $z$ denote the coordinates in a coordinate system centered at $y_0$ and with the $z_n$ axis pointing towards the inner normal to $\p \Om$.

In conclusion
$$|\nabla u_h(y_0)| \le C \quad \quad \Rightarrow \quad |\nabla u(x_0)|\le C_1 h^\frac{1+\alpha}{2+\alpha}, \quad \quad x_0=AF_h y_0,$$
and by choosing $h=c_2(\rho,\rho')$ small, we obtain the desired conclusion.

\end{proof}

From the proof above we see that if in Lemma \ref{l11} we have $\p \Om,\ph \in C^2$ in $B_\rho$ and
$$\det D^2 u=g \, d_{\p \Om}^\alpha,$$
for some function $g>0$ that is continuous on $\p \Om \cap B_\rho$,
then Theorem \ref{T2} applies at all points on $\p \Om \cap B_c$ with $c=c(\rho,\rho')$ small. In particular we obtain that $u$ is pointwise $C^2$ at all these points, and using the arguments above it can be shown that $D^2u$ is continuous on $\p \Om \cap B_c$.

\

Next we extend our estimates from $\p \Om$ to a small neighborhood of $\p \Om$ and prove Theorem \ref{T2.1}. 

{\it Proof of Theorem \ref{T2.1}}

In this proof we denote by $\bar c$, $\bar C$ various constants (that may change from line to line) which depend on $n$, $\alpha$, $\rho$, $\rho'$, $\beta$ and the $C^2$ modulus of continuity of $\ph$ and $\p \Om$.

Assume for simplicity that $D^2 \ph(0)=I$ and $g(0)=1$. We apply Theorem \ref{T2} and obtain that there exists a sliding $A$, with $|A| \le C(\rho,\rho')$, such that for any $\eta>0$
\begin{equation}\label{81}
(1-\eta) A \, \, S_h(U_0) \subset S_h(u) \subset (1+\eta) A \, \, S_h(U_0),
\end{equation}
for all $h \le \bar c(\eta)$.
 Let $t$ be the minimum value of $u-h^\frac{1+\alpha}{2+\alpha}x_n$ and $x_t$ the point where is achieved thus
 $$S_t(x_t)=\{ u<h^\frac{1+\alpha}{2+\alpha}x_n \}.$$
 Next we show that
 \begin{equation}\label{82}
 \|D^2 u\|_{C^\beta(S_{t/4}(x_t))} \le \bar C h^{-\frac \beta 2} , \quad \sup_{S_{t/4}(x_t)} \| D^2u-D^2u(0)\|\le \bar C \eta.
 \end{equation}

 From \eqref{81} we see that $t \sim h$ and also
 $$S_{t/2}(x_t) \subset \mathcal C:= \left \{|x'| \le C |x_n|   \right \}.$$
 In the cone $\mathcal C$, $d_{\p \Om}/x_n$ is a positive function with bounded Lipschitz norm, hence for all $t$ small
 $$\det D^2 u = \bar g \, \, x_n^\alpha \quad \quad \mbox{in $S_{t/2}(x_t)$}$$
 with $\bar g (0)=1$, $\|\bar g\|_{C^\beta} \le \bar C$.
 We let $u_h$ be the rescaled function given in \eqref{80} and let  $ x_t=A F_h y_t$ and
 notice that
 $$ A F_h  S_{t/(2h)}(y_t) = S_{t/2}(x_t), \quad \quad S_{t/h}(y_t)=\{u_h <y_n\}$$
where $S_t(y)$ denote the sections for $u_h$.

We have $$\det D^2 u_h= \bar g_h \, \, x_n^\alpha \quad \mbox{in}\quad S_{t/(2h)}(y_t),$$
with $$\bar g_h(y)=\bar g (AF_h y) \quad \quad \Rightarrow \quad \quad \|\bar g_h\|_{C^\beta} \le \bar C h^\frac {\beta}{2+\alpha}, \quad \bar g_h(0)=0.$$
From \eqref{81} we have $$|u_h -U_0| \le C \eta \quad \mbox{in} \quad  S_{t/(2h)}(y_t)$$
hence, by the interior $C^{2,\beta}$ estimates for Monge-Ampere equation, we obtain
$$\|D^2 u_h \|_{C^\beta} \le \bar C, \quad \|D^2_{x'} u_h - I \| \le \bar C \eta \quad \quad \mbox{in} \quad S_{t/(4h)}(y_t).$$
We write these inequalities in terms of $D^2 u$ and we obtain \eqref{82}. We apply the same argument at other boundary points instead of the origin, thus we may assume that \eqref{82} holds uniformly for all points $x^*\in \p \Om \cap B_\delta$ and their corresponding interior sections $S_{t'}(x^*_{t'})$ which become tangent to $\p \Om$ at $x^*$.

Let $y^* \in \p \Om_h \cap S_c(u_h)$, thus $|\nabla u_h(y^*)|$ is small if $c$ is small. This implies that the section
 $$S_{t'/h}(y^*_{t'}):=\{u_h<u_h(y^*) + (\nabla u_h(y^*) +\nu_{y^*}) \cdot (y-y^*)\},$$
with $\nu_{y^*}$ the inner normal to $\p \Om_h$, is a perturbation of the section $\{ u_h<y_n\}$. We obtain
  $$S_{t'/(4h)}(y^*_{t'}) \cap S_{t/(4 h)}(y_t) \ne \emptyset,$$
and the corresponding sections for $u$ satisfy
$$S_{t'/4}(x^*_{t'}) \cap S_{t/4}(x_t) \ne \emptyset,$$
if $x^*\in \p \Om \cap S_{ch}.$ This and \eqref{82} imply
$$\|D^2 u(x^*)-D^2u(0)\| \le \bar C \eta,$$
which together with \eqref{82} shows that $u \in C^2(\p \Om \cap B_\delta)$ for some small $\delta$.

\qed

\begin{rem}\label{r7}
From the proof above we see that if $g$ has a $C^\beta$ modulus of continuity only on $\p \Om$, i.e.
\begin{equation}\label{83}
|g(x)-g(x_0)| \le C |x-x_0|^ \beta  \quad \quad \mbox{for all $x\in \ov \Om$, $x_0\in \p \Om$,}
\end{equation}
then for $u \in C^{1,\gamma}(\ov \Om \cap B_\delta)$ for any $\gamma<1$, and with $\delta$ small depending also on $\gamma$.

Indeed, instead of the interior $C^{2,\beta}$ estimates we may apply the interior $C^{1,\gamma}$ estimates since $\bar g_h$ has small oscillation in $S_{t/(2h)}(y_t) $. We obtain
$$\|\nabla u_h \|_{C^\gamma} \le C \quad \quad \mbox{in} \quad S_{t/(4h)}(y_t), \quad \quad |\nabla u_h| \le C \quad \mbox{in} \quad S_1(u_h),$$
which rescaled back implies
$$\|\nabla u\|_{C^\gamma (S_{t/4}(x_t))} \le C h^\frac{1-\gamma}{2},  \quad \quad \quad \sup_{S_h}|\nabla u-\nabla u(0)| \le h^\frac 12,$$
and the claim easily follows.

\end{rem}

As a consequence of Theorem \ref{T2.1} we obtain Theorem \ref{T03}.

\

{\it Proof of Theorem \ref{T03}}

After multiplying by an appropriate constant we may suppose $\max_\Om|u|=1$. Since $\p \Om$ is uniformly convex, we can use explicit barriers at points on $\p \Om$ and obtain $|u| \le C d_{\p \Om}$  with $C$ a constant depending on $n$ and the lower bounds for the curvatures of $\p \Om$. Also by convexity we find $|u| \ge c d_{\p \Om}$.

These inequalities on $|u|$ imply that if $x_0 \in \p \Om$ then $c \le |\nabla u(x_0)| \le C$, hence on $\p \Om$ the function $u$ separates quadratically from its tangent plane at $x_0$. We apply Proposition \ref{p0} and obtain that $u$ is pointwise $C^{1,1/3}$ at all points on $\p \Om$, i.e.
$$0 \le u(x) - \nabla u(x_0) \cdot (x-x_0) \le C|x-x_0|^\frac 43  \quad \quad \mbox{for all $x\in \ov \Om$, $x_0\in \p \Om$.} $$
This implies $\nabla u \in C^{1/3}(\p \Om)$, which toghether with the inequality above gives that
$$g:= |u|/d_{\p \Om}$$
has a uniform $C^{1/3}$ modulus of continuity on $\p \Om$, i.e. \eqref{83} holds with $\beta=1/3$. By Remark \ref{r7} above we find $u \in C^{1,\gamma}(\ov \Om)$ which implies that $g \in C^\gamma(\ov \Om)$, and the conclusion follows by Theorem \ref{T2.1}.

\qed

Before we prove Theorem \ref{T02} we obtain a simple consequence of Thorem \ref{T3}. We recall the notation used in Section \ref{s4}
$$b(h):=\max_{S_h} x_n.$$

\begin{lem}\label{15}
For any $\eps>0$ small, there exist constants $\bar c$ small, $K$ large depending on $\mu$, $n$, $\alpha$ and $\eps$ such that if
$$u\in \mathcal D_0^\mu (a_1,...,a_{n-1}), \quad \quad \mbox{with} \quad a_{n-1} \ge K$$
and $\mu \le a_1 \le \cdots \le a_{n-1} \le \infty,$ then
$$b(t) \ge  (2 /\mu) \, t^\frac{1}{3+\alpha - \eps} \quad \quad \mbox{for some $t \in [\bar c,1]$}.$$
\end{lem}

\begin{proof}
In Theorem \ref{T3} we showed that if $0\le k \le n-2$,
 \begin{equation}\label{84}
 u \in \mathcal D_0^\mu(\underbrace{1,...,1}_{k \, \, times},\infty..,\infty)  \quad \quad \Rightarrow \quad b(h) \ge C h^\frac{1}{3+\alpha},
 \end{equation}
 for some universal $C$ depending on $\mu$, $n$, $\alpha$. Indeed, in Lemma \ref{l5} we obtained $c d_n \le b(h) \le C d_n$ and in \eqref{65.65}
 $$ch \le d_n^{n+1-k+\alpha} \le  C h, \quad \quad n+1-k+\alpha \ge 3 + \alpha.$$
 Now the lemma follows by compactness similar to the proof of Lemma \ref{l4.1}.
 From \eqref{84} with $k=0$ and by compactness, we can find $C_1(\eps)$ large such that the conclusion of the lemma holds if $a_1 \ge C_1$.

 If $a_1 \le C_1$ then we use compactness and \eqref{84} with $k=1$ (and $\tilde \mu$ depending on $\mu$ and $C_1$), and obtain that there exists $C_2(\eps)$, $C_2 \gg C_1$ such that if $a_2 \ge C_2$ then the conclusion of the lemma is satisfied.

 We obtain the conclusion by repeating this argument $n-2$ times.

\end{proof}

We conclude the section with the proof of Theorem \ref{T02}.

\

{\it Proof of Theorem \ref{T02}}

From Theorem \ref{T1} we know that after subtracting the tangent plane at the origin and after performing an affine deformation given by a sliding along $x_n=0$ we may suppose that
\begin{equation}\label{85}
 u=O(|x'|^2+x_n^{2+\alpha}), \quad \quad \mbox{near the origin.}
\end{equation}

For $h$ large we define as usually $d_1 \le ...\le d_{n-1}$ to be the length of the axis of the ellipsoid which is equivalent to $S_h \cap \{x_n=x^*_h \cdot e_n \}$, and we let $d_n$ such that
$$\left (\prod_1^{n-1} d_i^2\right) \, d_n^{2+\alpha}=h^n.$$
As in the proof of Proposition \ref{p1} we can find $c_0$, $C_0$ depending only an $n$ and $\alpha$, and a sliding $A_h$ along $x_n=0$ such that,
$$c_0 d_n \le b(h) \le C_0 d_n,$$
and the rescaling
$$u_h(x):=\frac 1 h u(A_h D_hx) \quad \quad \mbox{with} \quad D_h:=diag(d_1,..,d_n), $$
satisfies
$$u_h \in \mathcal D_0^{c_0}(a_1,..,a_{n-1}) \quad \quad \mbox{with} \quad a_i=d_i h^{-\frac 12} .$$

If
 \begin{equation}\label{86}
 b(h) \le c_1(\eps) h^{1/(2+\alpha)}
 \end{equation} 
 for some $c_1$ sufficiently small then $a_{n-1} \ge K$ with $K$ the constant from Lemma \ref{15} applied to $u_h$. Then
$$\frac{b(th)}{b(h)}=\frac{b_{u_h}(t)}{b_{u_h}(1)} \ge 2\, t^\frac{1}{3+\alpha-\eps} \quad \quad \mbox{for some $t\in [\bar c,1]$},$$
hence $$q(h) \le \frac 12 q(t h) \quad \quad \mbox{with} \quad q(h):=b(h) h^{-\frac{1}{3+\alpha-\eps}}.$$
Thus if \eqref{86} holds for all $h$ large then $q(h) \to 0$ as $h \to \infty$. This contradicts the growth assumption for $u$ at infinity on the $x_n$ axis.

In conclusion 
$$b(h) \ge c_1(\eps) h^\frac{1}{2+\alpha}$$
for a sequence $h=h_m$ tending to $\infty$, hence $c(\eps) \le d_ih^{-1/2} \le C(\eps)$ if $i < n$. This implies that for this sequence of $h_m$'s, the rescaled function
$$\tilde u_h(x):=\frac 1 h u(A_h F_hx) \quad \quad \mbox{with} \quad F_hx:=(h^\frac 12 x', h^\frac{1}{2+\alpha}), $$
satisfies the hypotheses of Theorem \ref{T2} for any $\eta>0$. Hence there exists $c_2(\eps,\eta)$ such that
$$(1-\eta)U_0 \le \tilde u_h(\tilde A_h x) \le (1+\eta) U_0(x) \quad \mbox{holds if $|x|\le c_2$,}  $$
for some sliding $\tilde A_h$. In terms of $u$ this means that
$$(1-\eta)U_0 \le u(\bar A_h x) \le (1+\eta)U_0 \quad  \mbox{holds if $|F_h^{-1}x|\le c_2$} ,$$
for some sliding $\bar A_h$. Using also \eqref{85} we obtain $\bar A_h=I$. We let $h \to \infty$, thus
$$(1-\eta)U_0 \le u \le (1+\eta) U_0 \quad \mbox{for all $x$,}$$
and, since $\eta$ is arbitrary, we find $u=U_0$.

\qed

\end{document}